\documentclass[journal,twoside]{IEEEtran}
\usepackage{cite}
\usepackage{amsmath,amssymb,amsfonts}
\usepackage{graphicx}
\usepackage{xcolor}
\usepackage{textcomp}
\usepackage{amsthm}
\usepackage{algorithm}
\usepackage{algpseudocode}

\theoremstyle{plain}
\newtheorem{theorem}{Theorem}
\newtheorem{lemma}{Lemma}

\newtheorem{corollary}{Corollary}

\theoremstyle{definition}
\newtheorem{assumption}{Assumption}
\newtheorem{remark}{Remark}

\makeatletter
\let\NAT@parse\undefined
\makeatother
\usepackage{hyperref}
\hypersetup{colorlinks=true,linkcolor=blue,citecolor=blue,urlcolor=blue}
\usepackage[nameinlink]{cleveref}

\crefname{assumption}{Assumption}{Assumptions}
\Crefname{assumption}{Assumption}{Assumptions}
\crefname{theorem}{Theorem}{Theorems}
\Crefname{theorem}{Theorem}{Theorems}
\crefname{lemma}{Lemma}{Lemmas}
\Crefname{lemma}{Lemma}{Lemmas}
\crefname{proposition}{Proposition}{Propositions}
\Crefname{proposition}{Proposition}{Propositions}
\crefname{corollary}{Corollary}{Corollaries}
\Crefname{corollary}{Corollary}{Corollaries}
\crefname{remark}{Remark}{Remarks}
\Crefname{remark}{Remark}{Remarks}
\crefname{section}{Section}{Sections}
\Crefname{section}{Section}{Sections}
\crefname{appendix}{Appendix}{Appendices}
\Crefname{appendix}{Appendix}{Appendices}
\crefname{figure}{Fig.}{Figs.}
\Crefname{figure}{Fig.}{Figs.}

\def\BibTeX{{\rm B\kern-.05em{\sc i\kern-.025em b}\kern-.08em
    T\kern-.1667em\lower.7ex\hbox{E}\kern-.125emX}}
\begin{document}
\title{Noise-Robust Distributed Optimization Over Directed Graphs With Row Stochastic Matrices}
\author{Yifan Wang, \IEEEmembership{Student Member, IEEE}, Mufeng Wang, and Xianghui Cao, \IEEEmembership{Senior Member, IEEE}
\thanks{Yifan Wang and Xianghui Cao are with the School of Automation,
Southeast University, Nanjing 210096, China (e-mail: evan@seu.edu.cn; xhcao@seu.edu.cn).}
\thanks{Mufeng Wang is with the China Industrial Control Systems Cyber Emergency Response Team, Beijing, 100040, China (wangmufeng1990@163.com).}}

\maketitle

\begin{abstract}
In the presence of information-sharing noise, row-stochastic distributed optimization over directed and unbalanced graphs can suffer not only from noise accumulation in gradient tracking, but also from distortion of the left eigenvector-based gradient scaling used for imbalance compensation. To address these issues, this paper proposes a noise-robust distributed optimization algorithm using only row-stochastic weights,
termed Robust Xi-row (R-Xi-row). The optimization update is driven by the increment of an auxiliary cumulative variable to exclude historical tracking noise, while two decaying gains and a normalized scaling gain are used to
suppress the effect of noisy eigenvector estimation and asymptotically recover the required gradient scaling almost surely. Under strongly convex and smooth objectives and appropriate decaying gains, we prove almost sure convergence to the optimal solution. For polynomially decaying gains, we further establish a near-$\mathcal{O}(k^{-1/3})$ expected convergence rate. Numerical experiments validate the theoretical results.
\end{abstract}


\section{Introduction}\label{sec:introduction}
\IEEEPARstart{E}{mpowered} by rapid advances in communication, sensing and computing technologies, we have witnessed stunning development in network systems, where various optimization problems arise across physically distributed and private datasets to achieve system-level global objectives. A canonical example is the cooperative optimization of a network of agents (potentially coordinated by a central node):
\begin{equation}\label{intro}
 \mathbf{x}^*\in\arg\min_{\mathbf{x}\in\mathbb{R}^d} F(\mathbf{x})\overset{\triangle}{=}\frac{1}{n}\sum_{i=1}^{n}f_i(\mathbf{x}),
\end{equation}
where $f_i(\cdot): \mathbb{R}^d\rightarrow\mathbb{R}$ is the local objective function of the $i$th agent. $\mathbf{x}\in\mathbb{R}^d$ is the optimization variable.
This problem is critical in a variety of fields including distributed learning \cite{luca}, cyber-physical systems \cite{yemini}, resource allocation \cite{zhaozeli}, smart grid \cite{TAC_WYF_An_Operator_Splitting_Scheme}, to name a few. {To address problem \eqref{intro}, distributed optimization has emerged as the preferred paradigm for solving problem \eqref{intro} compared to centralized schemes. This preference stems from the inherent limitations of centralized approaches, which suffer from single-point failure risks (due to reliance on a central coordinator) and communication bottlenecks caused by aggregating high-dimensional data across large-scale networks. As network sizes and variable dimensions grow exponentially, distributed schemes enable scalability, fault tolerance, and decentralized decision-making—critical attributes for modern distributed systems.} For the development of distributed optimization methods, distributed gradient descent (DGD) method has laid a significant foundation \cite{TAC_2009_Nedic_distributed_optimization,
TAC_Nedic_ContrainedConsensusAndOptimnizaiton_2010,
SIOPT_2016_Yuan_DGD,
TSP_2021_Berahas_NEAR_DGD}.

Most DGD-type algorithms rely on undirected or balanced
communication networks. Over a general directed and unbalanced graph,
however, the communication weight matrix is generally not doubly stochastic, and a
direct consensus iteration does not preserve the uniformly weighted
average required by \eqref{intro}. Therefore, different mechanisms have been developed to compensate for the resulting network imbalance.
Subgradient-push and Push-DIGing employ column-stochastic weights
together with push-sum normalization
\cite{TAC_2015_subgradientpush,PUSH_DIGING}, whereas Push--Pull and
its variants combine row- and column-stochastic weights
\cite{push_pull,TCNS2024_Accelerated_pushpull,WANG2025111925,arxiv26_1}.
Another line of work relies only on row-stochastic weights and
estimates the left Perron eigenvector\footnote{Throughout this paper, we simply refer to it as left eigenvector estimate.} of the communication weight matrix to correct
the resulting gradient bias
\cite{Automatica_row_stochastic_Mai,
Linear_convergence_in_optimization_over_directed_graphs_with_row-stochastic_matrices,
TAC_2024_Fast_row_stochastic,xirow+1,arxiv26_2}.
Compared with column-stochastic schemes, these row-stochastic methods
allow each agent to assign weights to the information received from
its in-neighbors without requiring its out-degree information \cite{adhoc_2}. In this
paper, we refer to this class of algorithms as Xi-row methods.
\par The aforementioned studies are built on the exact information exchange, i.e., each agent acquires its in-neighbors' information without any distortion. In practical networks, however, the transmission
may be corrupted by channel noise \cite{jsac_channel_noise},
quantization errors
\cite{TAC_QuantizedDistributedGradientTrackingAlgorithmWithLinearConvergenceinDirectedNetworks},
or deliberately injected privacy noise
\cite{TCNS_ADifferentiallyPrivateMethodforDistributedOptimizationinDirectedNetworksviaStateDecomposition}. To attenuate such perturbations, combining the algorithms with the decaying gains provides an effective means \cite{TSP_2019_An_exact_quantized_decentralized_gradient_descent_algorithm,H.Reisizadeh_2023TAC_Distributedoptimizationovertimevaryinggraphswithimperfectsharingofinformation,TAC_2024_WYQ_Tailoring_Gradient_Methods_for_Differentially_Private_Distributed_Optimization,decaying+1}. For gradient-tracking algorithms, persistent information-sharing noise is more involved because the perturbations entering the gradient-tracking recursion may be retained through its temporal evolution, thereby destroying exact average-gradient tracking. Consequently, robust gradient-tracking schemes have been developed for directed networks. For example, within a Push-Pull-type
architecture, the method in
\cite{CDC_A_Robust_Gradient_Tracking_Method_for_Distributed_Optimization_over_Directed_Networks}
attenuates the accumulated noise effect but generally admits a
steady-state optimization error. Subsequent works further combine
robust gradient tracking with decaying gains to recover exact
optimality under persistent information-sharing noise
\cite{decaying+2,decaying+3,
TAC_GradientTrackingBasedDistributedOptimizationWithGuaranteedOptimalitUnderNoisyInformationSharing}. These methods are developed within Push-Pull-type architectures that
combine row- and column-stochastic mixing. 

By contrast, in the purely row-stochastic Xi-row setting, persistent information-sharing noise introduces an additional
difficulty because the imbalance compensation relies on an online estimate of the left
eigenvector. More specifically, two different errors arise. First,
directly incorporating decaying gains into the conventional Xi-row
gradient-tracking recursion retains the accumulated tracking noise, so
the average gradient cannot be recovered. Second, when the exchanged
left-eigenvector estimates are also noisy, the gradient scaling used
to compensate for the bias induced by row-stochastic mixing becomes distorted.

In this paper, we focus on additive zero-mean and variance-bounded
noise, while the noise affecting the left-eigenvector estimation is
additionally assumed to be uniformly bounded. To address the resulting
coupled effects, we develop a noise-robust row-stochastic distributed
optimization algorithm, termed Robust Xi-row (R-Xi-row). R-Xi-row
redesigns both the gradient-tracking and eigenvector-estimation
mechanisms to exclude historical tracking noise from the effective gradient descent direction and to enable asymptotic recovery of the gradient scaling required for row-stochastic imbalance compensation under noisy left-eigenvector estimation.

1) We characterize the two noise effects specific to row-stochastic
distributed optimization. A direct diminishing-mixing extension of
Xi-row retains the accumulated tracking-noise term even when the
eigenvector estimator is noise-free. When the eigenvector-estimation
dynamics are further corrupted, an additional scaling error enters the
imbalance-compensated gradient direction.

2) We develop the R-Xi-row algorithm using only row-stochastic weights. The increment of an auxiliary cumulative variable is employed as the descent direction to exclude historical tracking noise from the optimization update. For uniformly bounded noise in the  left eigenvector estimation, two decaying gains with different decay rates make the accumulated noise contribution asymptotically negligible relative to the injected eigenvector correction terms, while a normalized scaling gain recovers the required gradient scaling.

3) Under strongly convex and $L$-smooth objectives,
together with the stated assumptions on information-sharing noise and
decaying gains, we prove almost sure convergence to the optimal
solution. For polynomially decaying gains, we further establish an 
$\mathcal O((k+1)^{-m})$ expected convergence rate for the optimality 
gap and squared consensus and tracking errors. The admissible gain 
exponents can be selected such that $m$ approaches $\frac{1}{3}$ 
arbitrarily closely, yielding a near-$\mathcal O(k^{-\frac{1}{3}})$
expected convergence rate. Among existing noise-robust
gradient-tracking studies over directed graphs,
\cite{TAC_GradientTrackingBasedDistributedOptimizationWithGuaranteedOptimalitUnderNoisyInformationSharing}
is particularly relevant to our setting because it also considers
persistent information-sharing noise and incorporates online
left-eigenvector estimation. It establishes exact almost sure
convergence but does not provide an explicit nonasymptotic rate.

The remainder of this paper is organized as follows: We formulate the problem in Section~\ref{sec:problem}. The proposed algorithm is presented in Section~\ref{sec:algorithm}. In Section~\ref{sec:convergence}, we present the convergence analysis of the R-Xi-row algorithm. Numerical experiments are given in Section~\ref{sec:numerical}. Finally, we conclude this paper in Section~\ref{sec:conclusion}.

\textit{Notation:} Throughout this paper, we introduce the following conventions: Vectors and scalars are represented by bold and normal lowercase letters, respectively. $\mathbb{N}$ and $\mathbb{R}_{+}$ denote the set of nonnegative integers and nonnegative real numbers. 
$\mathbb{R}^{d}$ and $\mathbb{R}^{n\times d}$ denote the sets of real vectors of dimension $d$, and real matrices of dimension $n\times d$, respectively. The identity matrix and the vector of ones are represented by $I$ and $\mathbf{1}$ with proper dimensions, respectively. The superscript $\top$ denotes the transpose of a vector or matrix. $\langle\cdot,\cdot\rangle$ denotes the inner product. We denote $\|\mathbf{x}\|$ as the Euclidean norm of a vector $\mathbf{x}$. When the subscript is explicitly shown,
$\|\cdot\|_2$ denotes the Euclidean norm for vectors and the induced
$2$-norm for matrices. The stochastic expectation is denoted by $\mathbb{E}[\cdot]$.

\section{Problem Statement}\label{sec:problem}
In this section, we formulate the distributed optimization problem under information-sharing noise and show that naively incorporating diminishing mixing into the conventional Xi-row algorithm does not remove the effect of such noise.

\subsection{Problem Formulation}
Consider a set of $n$ agents interacting over a directed graph $\mathcal{G}=(\mathcal{V},\mathcal{E})$, where $\mathcal{V}=\{1,2,...,n\}$ denotes the set of agents, and $\mathcal{E}\subseteq\{(i,j): i,j\in\mathcal{V}\}$ denotes the set of ordered pairs representing information-exchange links among agents. If $(i,j)\in\mathcal{E}$, it means that agent $j$ can send its local information to agent $i$, or equivalently, agent $j$ is an in-neighbor of agent $i$. We denote the collection of agent $i$ itself and all of its in-neighbors by $\mathcal{N}_i^{in}$.

The objective is to minimize a global function $F(\mathbf{x}):=\frac{1}{n}\sum_{i=1}^{n}f_i(\mathbf{x})$ as defined in \eqref{intro}. Since $\mathbf{x}$ is a centralized variable, to make it distributively tractable, we equivalently rewrite the problem as follows:
\begin{eqnarray}\label{problem}
\min\limits_{\mathbf{x}_1,\mathbf{x}_2,\cdots,\mathbf{x}_n\in\mathbb{R}^d}\,\,\, \hspace{-7pt}&&\hspace{-7pt}\frac{1}{n}\sum_{i=1}^{n}f_i(\mathbf{x}_i)\nonumber\\ \qquad\,\text{s.t.}\quad \qquad \hspace{-7pt}&&\hspace{-7pt}\mathbf{x}_1=\mathbf{x}_2=\cdots=\mathbf{x}_n
\end{eqnarray}
where $\mathbf{x}_i\in\mathbb{R}^d$ is a local copy of $\mathbf{x}$ for $i=1,2,\cdots,n.$

To reach consensus on an optimal optimization solution $\mathbf{x}^*$ in distributed fashion, agents exchange information, e.g., local estimate of the optimal solution, with their neighbors over the networks. Suppose that each agent $i\in\mathcal{V}$ receives exact information from its in-neighbors, and performs, e.g., distributed gradient descent, to update a local estimate of an optimal solution of \eqref{problem} at each step $k$ as follows:
\begin{eqnarray}\label{dgd}
  \mathbf{x}_{i,k+1}\hspace{-7pt}&=&\hspace{-7pt}\sum_{j=1}^{n}W_{ij}\mathbf{x}_{j,k}-\alpha_k\nabla f_i(\mathbf{x}_{i,k}),
\end{eqnarray}
where $\alpha_k$ is a decaying step size. $W_{ij}$ is the weight of the ordered pair $(i,j)$ defined as $W_{ij}\in(0,1)$ if $j\in\mathcal{N}_i^{in}$, or otherwise $W_{ij}=0$. 
However, when the exchanged information is noisy, it has access to an imperfect weighted average of its neighbors's states at each step $k$, denoted by $\hat{\mathbf{x}}_{i,k}$. More precisely, it reads as\footnote{A more precise model is $\hat{\mathbf{x}}_{i,k}
=\sum_{j=1}^{n}W_{ij}\big(\mathbf{x}_{j,k}+\mathbf{e}_{ij,k}^{x}\big)$ where $\mathbf{e}_{ij,k}^{x}$ denotes the information-sharing noise associated with the transmission from agent $j$ to agent $i$. 
For notational simplicity, we adopt the equivalent aggregated representation as in \eqref{noisy_state}.}
\begin{equation}\label{noisy_state}
  \hat{\mathbf{x}}_{i,k}=\sum_{j=1}^{n}W_{ij}\mathbf{x}_{j,k}+\mathbf{e}_{i,k}^x,
\end{equation}
where $\mathbf{e}_{i,k}^x\in\mathbb{R}^d$ is the information-sharing noise. In the sequel, $\mathbf{\hat{x}}_{i,k}$ is referred to as the noisy mixing state.

Throughout this paper, we impose the following standard assumptions commonly used in the literature \cite{Automatica_row_stochastic_Mai,Linear_convergence_in_optimization_over_directed_graphs_with_row-stochastic_matrices,TAC_2024_Fast_row_stochastic,xirow+1,arxiv26_2,TAC_QuantizedDistributedGradientTrackingAlgorithmWithLinearConvergenceinDirectedNetworks,H.Reisizadeh_2023TAC_Distributedoptimizationovertimevaryinggraphswithimperfectsharingofinformation,decaying+2,wyf,2021_pushu_TAC_push-pull}.

\begin{assumption}[Network assumption]\label{ass:network}
$W=[W_{ij}]$ is a primitive nonnegative matrix and satisfies $W\mathbf{1}=\mathbf{1}$.
\end{assumption}

Under Assumption~\ref{ass:network}, the row-stochastic weight matrix $W$ admits a unique positive left eigenvector $\mathbf r$ associated with the eigenvalue~1, normalized by $\mathbf r^\top \mathbf 1=1$. The vector $\mathbf r$ can be interpreted as the stationary distribution induced by $W$, in the sense that $\mathbf r^\top W=\mathbf r^\top$ and $W^k$ converges to $\mathbf 1\mathbf r^\top$.

\begin{assumption}[Problem assumption]\label{ass:problem}
For each $i\in\mathcal{V}$, the local objective function $f_i$ is $\mu$-strongly convex and $L$-smooth for some $\mu, L>0$.
\end{assumption}



To suppress the information-sharing noise, we introduce another decaying gain $\beta_k$ to the noisy mixing state $\hat{\mathbf{x}}_{i,k}$. Specifically, combining diminishing mixing with gradient descent, we consider the following two-timescale update of
\begin{eqnarray}\label{diminishingmixing2}
  \mathbf{x}_{i,k+1} \hspace{-7pt} &=& \hspace{-7pt}\underbrace{(1-\beta_k)\mathbf{x}_{i,k}+\beta_k\hat{\mathbf{x}}_{i,k}}_{\text{Diminishing mixing}}\underbrace{-\alpha_k\nabla f_i(\mathbf{x}_{i,k})}_{\text{Gradient descent}}.
\end{eqnarray}
The first two terms on the right-hand side of \eqref{diminishingmixing2} constitute the diminishing-mixing mechanism for in-neighbors' information, while the last term performs gradient descent using local objective. It is well known that, in such a two-timescale framework, $\beta_k$ needs to be larger than $\alpha_k$ to ensure the convergence \cite{H.Reisizadeh_2023TAC_Distributedoptimizationovertimevaryinggraphswithimperfectsharingofinformation,TAC_2021_T.Doan_ConvergenceRatesofDistributedGradientMethodsUnderRandomQuantization}. This means that the sum that mixes the local estimate of optimal solution with the noisy mixing state, i.e., $(1-\beta_k)\mathbf{x}_{i,k}+\beta_k\hat{\mathbf{x}}_{i,k}$, changes more dramatically from step-by-step, and hence operates on the ``fast" timescale. Meanwhile, the gradient descent step behaves more conservatively, and so operates on the ``slow" timescale. 
\subsection{Naively Incorporating Diminishing Mixing into Xi-row Does Not Work}
We first recall the original Xi-row algorithm in \cite{Linear_convergence_in_optimization_over_directed_graphs_with_row-stochastic_matrices}: To solve problem \eqref{problem}, each agent $i\in\mathcal{V}$ exchanges three vectors $\mathbf{x}_{i,k}$, $\mathbf{y}_{i,k}$ and $\mathbf{z}_{i,k}$ with its neighbors at each step $k\in\mathbb{N}$, and performs the updates:
\begin{subequations}\label{xirow}
    \begin{eqnarray}
    \mathbf{x}_{i,k+1}\hspace{-7pt}&=&\hspace{-7pt}\sum_{j=1}^{n}W_{ij}\mathbf{x}_{j,k}-\tau \mathbf{z}_{i,k},\label{xirow_a}\\
    \mathbf{y}_{i,k+1}\hspace{-7pt}&=&\hspace{-7pt}\sum_{j=1}^{n}W_{ij}\mathbf{y}_{j,k},\label{xirow_b}\\
    \mathbf{z}_{i,k+1}\hspace{-7pt}&=&\hspace{-7pt}\sum_{j=1}^{n}W_{ij}{\mathbf{z}}_{j,k}+\frac{\nabla f_i(\mathbf{x}_{i,k+1})}{[\mathbf{y}_{i,k+1}]_i}-\frac{\nabla f_i(\mathbf{x}_{i,k})}{[\mathbf{y}_{i,k}]_i},\label{xirow_c}
    \end{eqnarray}
\end{subequations}
where $\tau$ is a constant step size for gradient descent. The algorithm is initialized with $\mathbf{x}_{i,0}\in\mathbb{R}^d$, $\mathbf{y}_{i,0}=[0,...,1,...,0]^{\top}$ where the only nonzero entry is the $i$-th one and $\mathbf{z}_{i,0}=\nabla f_i(\mathbf{x}_{i,0})$. $[\mathbf{y}_{i,k}]_i$ denotes the $i$-th entry of the vector $\mathbf{y}_{i,k}$.

In essence, the updates of \eqref{xirow_a} and \eqref{xirow_c} form a modified version of gradient tracking in \cite{2018_QU_TCNS_Harnessing_Smoothness_to_Accelerate_Distributed_Optimization}, where the gradient differences are scaled by the iterates generated through \eqref{xirow_b}. Note that, according to the Perron-Frobenius theorem \cite{nonnegative_book}, the update in \eqref{xirow_b} converges exponentially to the left eigenvector $\mathbf{r}^{\top}$ of the row-stochastic weight matrix $W$ with respect to (w.r.t.) eigenvalue 1. Hence, $\mathbf{y}_{i,k}\in\mathbb{R}^n$ serves as a local estimate of this left eigenvector at agent $i$. Consequently, scaling the gradients by the iterates generated from \eqref{xirow_b}, as implemented in \eqref{xirow_c}, offsets the imbalance caused by relying solely on row-stochastic weight matrix.

We next show that diminishing mixing, when naively incorporated into Xi-row, fails to eliminate the effect of information-sharing noise. To make it clear, the left eigenvector estimate $\mathbf{y}_{i,k}$ is assumed to be noise-free. This simplification is without loss of the main insight, since the presence of information-sharing noise in the updates of \eqref{xirow_a} and \eqref{xirow_c}, or even in \eqref{xirow_c} alone, is already sufficient to reveal this point.

A natural way for the Xi-row algorithm to handle information-sharing noise is to mimic the two-timescale update of \eqref{diminishingmixing2}. Mathematically, it yields the following modifications:
\begin{subequations}\label{naive combination}
    \begin{eqnarray}
    \mathbf{x}_{i,k+1}\hspace{-7pt}&=&\hspace{-7pt}(1-\beta_k)\mathbf{x}_{i,k}+\beta_k\hat{\mathbf{x}}_{i,k}-\tau \mathbf{z}_{i,k}\label{naive combination a},\\
    \mathbf{y}_{i,k+1}\hspace{-7pt}&=&\hspace{-7pt}\sum_{j=1}^nW_{ij}\mathbf{y}_{j,k}\label{naive combination b},\\
    \mathbf{z}_{i,k+1}\hspace{-7pt}&=&\hspace{-7pt}(1-\beta_k)\mathbf{z}_{i,k}+\beta_k\hat{\mathbf{z}}_{i,k}+\frac{\nabla f_i(\mathbf{x}_{i,k+1})}{[\mathbf{y}_{i,k+1}]_i}-\frac{\nabla f_i(\mathbf{x}_{i,k})}{[\mathbf{y}_{i,k}]_i},\nonumber\\
    &&\label{naive combination c}
    \end{eqnarray}
\end{subequations}
where $\hat{\mathbf{z}}_{i,k}=\sum_{j=1}^{n}W_{ij}\mathbf{z}_{j,k}+\mathbf{e}_{i,k}^{z}$ with $\mathbf{e}_{i,k}^z\in\mathbb{R}^d$ denoting the information-sharing noise appearing in the update of $\mathbf{z}_{i,k}$. The initialization of \eqref{naive combination} is the same as that of \eqref{xirow}.

For notational simplicity, this paper considers the scalar case $x_{i,k}$, $z_{i,k}\in\mathbb{R}$. Accordingly, it follows with $e_{i,k}^x$, $e_{i,k}^z\in\mathbb{R}$. The analysis can be extended to arbitrary dimension $\mathbb{R}^d$ by using the Kronecker product \cite{TAC_2024_Fast_row_stochastic}. 
We introduce the following compact notations:
\begin{equation*}
    \begin{aligned}
    \mathbf{x}_{k} &=[x_{1,k},\ldots,x_{n,k}]^{\top}\in\mathbb{R}^{n},\quad \mathbf{z}_{k} =[z_{1,k},\ldots,z_{n,k}]^{\top}\in\mathbb{R}^{n},\\
    Y_{k}          &=[\mathbf{y}_{1,k},\ldots,\mathbf{y}_{n,k}]^{\top}\in\mathbb{R}^{n\times n},\quad
    \tilde{Y}_{k}  =\mathrm{diag}(Y_{k})\in\mathbb{R}^{n\times n}
    \end{aligned}
\end{equation*}
\begin{equation*}
    \begin{aligned}
    \nabla\mathbf{f}_{k} &=[\nabla f_1(x_{1,k}),\ldots,\nabla f_n(x_{n,k})]^{\top}\in\mathbb{R}^{n},\\
    \nabla\mathbf{f}^*&=[\nabla f_1(x^*),\ldots,\nabla f_n(x^*)]^{\top}\in\mathbb{R}^{n},\\
    \mathbf{e}_{k}^{x} &=[e_{1,k}^{x},\ldots,e_{n,k}^{x}]^{\top}\in\mathbb{R}^{n},\quad \mathbf{e}_{k}^{z} =[e_{1,k}^{z},\ldots,e_{n,k}^{z}]^{\top}\in\mathbb{R}^{n},\\
    \mathbf{e}_{k}^{y} &=[\mathbf{e}_{1,k}^{y},\ldots,\mathbf{e}_{n,k}^{y}]^{\top}\in\mathbb{R}^{n\times n},
    \end{aligned}
\end{equation*}
where $\mathbf{e}_{i,k}^y\in\mathbb{R}^n$ denotes the information-sharing noise appearing in \eqref{re_3a}. The corresponding noisy mixing state of local eigenvector estimation at agent $i$ is given by $\hat{\mathbf{y}}_{i,k}=\sum_{j=1}^{n}W_{ij}\mathbf{y}_{j,k}+\mathbf{e}_{i,k}^y$.

Let's denote the diminishing mixing matrix as 
$${W}_k:{=}(1-\beta_k)I+\beta_kW.$$
Under Assumption~\ref{ass:network}, $W_k$ is also row stochastic and has the same positive left eigenvector $\mathbf{r}$ w.r.t. eigenvalue one. 

The compact form of \eqref{naive combination b} and \eqref{naive combination c} is given by
\begin{eqnarray}
    Y_{k+1}\hspace{-7pt}&=&\hspace{-7pt}{W}Y_k,\label{naive_combin_houmian1a}\\
    \mathbf{z}_{k+1}\hspace{-7pt}&=&\hspace{-7pt}{W}_k\mathbf{z}_k+\beta_k\mathbf{e}_{k}^z+\tilde{Y}_{k+1}^{-1}\nabla \mathbf{f}_{k+1}-\tilde{Y}_k^{-1}\nabla \mathbf{f}_k.\label{naive_combin_houmian1b}
\end{eqnarray}
By premultiplying \eqref{naive_combin_houmian1b} with $\mathbf{r}^{\top}$ and telescoping from $t=0$ to $k$, one can obtain that
\begin{eqnarray}\label{not work}
 \hspace{-10pt}{\mathbf{r}^{\top}} \mathbf{z}_{k\hspace{-1pt}+\hspace{-1pt}1}
  \hspace{-8pt}&=&\hspace{-8pt}\sum_{t=0}^{k}\hspace{-1pt}{\beta_t}\mathbf{r}^{\top}\mathbf{e}_t^z\hspace{-1pt}+\hspace{-1pt}{\mathbf{r}^{\top}}\left(\text{diag}\left({W}^{k+1}Y_0\right)\right)^{\hspace{-1pt}-1}\hspace{-1pt}\nabla \mathbf{f}_{k\hspace{-1pt}+\hspace{-1pt}1}\hspace{-1pt}.
\end{eqnarray}
If the information-sharing is noise-free, i.e., the first term on the right-hand side of \eqref{not work} is absent, taking the limit on both sides of \eqref{not work} yields
\begin{eqnarray}\label{not work_if_noisefree}
   \hspace{-10pt} \lim\limits_{k\rightarrow\infty}{\mathbf{r}^{\top}}\mathbf{z}_{k}\hspace{-7pt}&=&\hspace{-7pt}\lim\limits_{k\rightarrow\infty}{\mathbf{r}^{\top}}\text{diag}\left({W}^{k}Y_0\right)^{-1}\nabla \mathbf{f}_{k}=\mathbf{1}^{\top}\nabla \mathbf{f}_{\infty},
\end{eqnarray}
where the second equality follows from $\lim_{k\rightarrow\infty}{W}^k={\mathbf{1}\mathbf{r}^{\top}}$ and $Y_0=I$. If the noise $\mathbf{e}_k^z$ persistently exists over all the steps $k$, the information-sharing noise $\mathbf{e}_k^z$ is accumulated over the entire history. Hence, no matter whether the noise $\mathbf{e}_k^z$ is suppressed by the decaying gain $\beta_k$ as $k$ goes to infinity, $\mathbf{z}_{k}$ is unable to accurately track the aggregated gradient $\mathbf{1}^{\top}\nabla \mathbf{f}_{k}$ by the conventional gradient-tracking dynamics \eqref{xirow_c}.
Similar observations can be found in prior work on Push-Pull algorithm
\cite{TAC_QuantizedDistributedGradientTrackingAlgorithmWithLinearConvergenceinDirectedNetworks,TCNS_ADifferentiallyPrivateMethodforDistributedOptimizationinDirectedNetworksviaStateDecomposition,CDC_A_Robust_Gradient_Tracking_Method_for_Distributed_Optimization_over_Directed_Networks,TAC_GradientTrackingBasedDistributedOptimizationWithGuaranteedOptimalitUnderNoisyInformationSharing}.

However, the present row-stochastic setting is fundamentally more challenging. In particular, the correction of the intrinsic imbalance introduced by row-stochastic mixing relies on the online left eigenvector estimate. Once information-sharing noise contaminates this recursion, the
gradient scaling gain itself becomes noisy, thereby introducing an
additional error into the imbalance compensation. As a result, the
existing convergence results for robust Push--Pull-type methods do not
directly cover the purely row-stochastic setting considered here,
where the online left-eigenvector estimation is also affected by
information-sharing noise. 

\section{The Proposed Algorithm: R-Xi-row}\label{sec:algorithm}
Motivated by the aforementioned two noise effects in Xi-row-type algorithms, we redesign both the 
gradient-tracking and left-eigenvector estimate dynamics. The resulting algorithm, termed R-Xi-row, is given by 
\begin{subequations}\label{re_3}
       \begin{eqnarray}
      \hspace{-7pt} \mathbf{y}_{i,k+1}\hspace{-7pt}&=&\hspace{-7pt}(1-\lambda_k)\mathbf{y}_{i,k}+\lambda_k\hat{\mathbf{y}}_{i,k}+\gamma_k\mathbf{y}_{i,0},\label{re_3a}\\
       \hspace{-7pt}  {z}_{i,k+1} \hspace{-7pt}&=&\hspace{-7pt}(1-\beta_k){z}_{i,k}+\beta_k\hat{{z}}_{i,k}+\alpha_k \kappa_{i,k}{\nabla f_i({x}_{i,k})},\label{re_3b}\\
       \hspace{-7pt}  {x}_{i,k+1} \hspace{-7pt}&=&\hspace{-7pt}(1-\beta_k){x}_{i,k}+\beta_k\hat{{x}}_{i,k}- ({z}_{i,k+1}-{z}_{i,k}).\label{re_3c}
       \end{eqnarray}
\end{subequations}

The two modifications in \eqref{re_3} address different effects of
information-sharing noise. In \eqref{re_3b}--\eqref{re_3c},
$z_{i,k}$ serves as an auxiliary cumulative variable, while the
optimization update is driven by its increment
$z_{i,k+1}-z_{i,k}$ rather than by $z_{i,k}$ itself. As shown in
Remark~\ref{remark:historical_noise}, this incremental structure
prevents historical tracking noise from directly entering the
descent direction. For the left-eigenvector estimation in \eqref{re_3a}, two decaying
gains $\lambda_k$ and $\gamma_k$ with different decay rates are
introduced. The gain $\lambda_k$ associated with noisy mixing is
required to decay faster than the correction gain $\gamma_k$.
Consequently, the accumulated noise contribution becomes
asymptotically negligible relative to the repeatedly injected
eigenvector information. Since this correction also changes the
magnitude of the left-eigenvector estimate, the scaling gain $\kappa_{i,k}$ in \eqref{re_3b} is designed as follows:
\begin{equation}\label{eq:safeguarded_kappa}
\kappa_{i,k}
=
\begin{cases}
\displaystyle
\frac{\sum_{j=1}^{n}[\mathbf y_{i,k+1}]_j}
{n[\mathbf y_{i,k+1}]_i},
& \text{if}\quad [\mathbf y_{i,k+1}]_i>\varepsilon_{\kappa},\\[3mm]
\kappa_{i,k-1},
& \text{otherwise},
\end{cases}
\end{equation}
where $\varepsilon_{\kappa}>0$ is a prescribed small constant to avoid division by a denominator close to zero, and $\kappa_{i,k}$ is initialized by
$\kappa_{i,-1}=0$. The asymptotic convergence of $\kappa_{i,k}$ is proved
in Lemma~\ref{y_nr}. A complete implementation of R-Xi-row is provided in Algorithm~\ref{alg:R-Xi-row}.

     \begin{algorithm}[t]
    \caption{R-Xi-row algorithm}
    \label{alg:R-Xi-row}
    \begin{algorithmic}[1]
    
    \State \textbf{Input:} Decaying gains
    $\{\lambda_k\}$, $\{\gamma_k\}$, $\{\alpha_k\}$ and $\{\beta_k\}$,
    safeguard threshold $\varepsilon_{\kappa}>0$.
    
    \State \textbf{Initialize:}
    $x_{i,0}\in\mathbb{R}^d$,
    $\mathbf{y}_{i,0}=[0,\ldots,1,\ldots,0]^{\top}$,
    where only the $i$-th entry is one,
    $z_{i,0}\in\mathbb{R}^d$, and $\kappa_{i,-1}=0$,
    $\forall i\in\mathcal V$.
    
    \State \textbf{For} $k=0,1,\ldots$, each agent $i\in\mathcal V$ repeats the following steps:
    
    \Statex \qquad
    Form the noisy mixing states:
    
    \Statex \qquad\qquad
    $\displaystyle
    \hat{\mathbf y}_{i,k}
    \leftarrow
    \sum_{j=1}^{n}W_{ij}\mathbf y_{j,k}
    +\mathbf e^y_{i,k};
    $
    
    \Statex \qquad\qquad
    $\displaystyle
    \hat z_{i,k}
    \leftarrow
    \sum_{j=1}^{n}W_{ij}z_{j,k}
    +e^z_{i,k};
    $
    
    \Statex \qquad\qquad
    $\displaystyle
    \hat x_{i,k}
    \leftarrow
    \sum_{j=1}^{n}W_{ij}x_{j,k}
    +e^x_{i,k};
    $
    
    \Statex \qquad
    $\displaystyle
    \mathbf y_{i,k+1}
    \leftarrow
    (1-\lambda_k)\mathbf y_{i,k}
    +\lambda_k\hat{\mathbf y}_{i,k}
    +\gamma_k\mathbf y_{i,0};
    $
    
    \Statex \qquad
    \textbf{If}
    $[\mathbf y_{i,k+1}]_i>\varepsilon_{\kappa}$
    \textbf{then}
    
    \Statex \qquad\qquad
    $\displaystyle
    \kappa_{i,k}
    \leftarrow
    \frac{\sum_{j=1}^{n}[\mathbf y_{i,k+1}]_j}
    {n[\mathbf y_{i,k+1}]_i};
    $
    
    \Statex \qquad
    \textbf{else}
    
    \Statex \qquad\qquad
    $\displaystyle
    \kappa_{i,k}\leftarrow\kappa_{i,k-1};
    $
    \Statex \qquad \textbf{End if}
    \Statex \qquad
    $\displaystyle
    z_{i,k+1}
    \leftarrow
    (1-\beta_k)z_{i,k}
    +\beta_k\hat z_{i,k}
    +\alpha_k\kappa_{i,k}\nabla f_i(x_{i,k});
    $
    
    \Statex \qquad
    $\displaystyle
    x_{i,k+1}
    \leftarrow
    (1-\beta_k)x_{i,k}
    +\beta_k\hat x_{i,k}
    -(z_{i,k+1}-z_{i,k});
    $
    \State \textbf{End for}
    \State \textbf{Output:} $\{x_{i,k}\}$.
    \end{algorithmic}
    \end{algorithm}

\begin{remark}[Exclusion of historical tracking noise]\label{remark:historical_noise}
  The convergence analysis will establish $\kappa_{i,k}\rightarrow \frac{1}{nr_i}$ in the next section. To isolate the role of the modified gradient tracking dynamics, we temporarily set
$\kappa_{i,k}=\frac{1}{nr_i}$. Then, the compact form of
\eqref{re_3b} becomes
    \begin{eqnarray}
  \mathbf{z}_{k+1}\hspace{-8pt}&=&\hspace{-8pt}W_k\mathbf{z}_k + \beta_k\mathbf{e}_k^z+\frac{\alpha_k}{n}R^{-1}\nabla \mathbf{f}_k,\label{algorithmb}
    \end{eqnarray}
    where $R=\text{diag}\{r_i\}_{i=1}^n$. 
Let $\Delta \mathbf{z}_{k+1}:=\mathbf{z}_{k+1}-\mathbf{z}_k$. By multiplying $\mathbf{r}^{\top}$ on both sides of \eqref{algorithmb}, we obtain
    \begin{eqnarray}\label{2rd_11}
       \mathbf{r}^{\top}\Delta \mathbf{z}_{k+1} \hspace{-7pt}&=&\hspace{-7pt}{\beta_k}\mathbf{r}^{\top}\mathbf{e}_{k}^z+\frac{\alpha_k}{n}{\mathbf{1}^{\top}}\nabla \mathbf{f}_k.
    \end{eqnarray}
    Comparing \eqref{2rd_11} with \eqref{not work}, the conventional
    Xi-row tracking state $\mathbf z_{k+1}$ contains the accumulated
    noise $
    \sum_{t=0}^{k}\beta_t\mathbf r^\top\mathbf e_t^z$,
    whereas the increment $\Delta\mathbf z_{k+1}$ employed by
    R-Xi-row contains only the current perturbation
    $\beta_k\mathbf r^\top\mathbf e_k^z$. Hence, the historical
    tracking noise stored in the cumulative state does not directly
    enter the descent direction in \eqref{re_3c}.
\end{remark}

\begin{remark}[Comparison with state-of-the-art methods]
    Conventional Xi-row-type methods compensate for directed-graph
    imbalance through online left-eigenvector estimation, while their
    original convergence results are established under accurate
    information exchange
    \cite{Linear_convergence_in_optimization_over_directed_graphs_with_row-stochastic_matrices,
    TAC_2024_Fast_row_stochastic}.
    Recent robust gradient-tracking methods allow persistent
    information-sharing noise to affect the optimization and tracking
    variables. In particular,
    \cite{TAC_GradientTrackingBasedDistributedOptimizationWithGuaranteedOptimalitUnderNoisyInformationSharing}
    further incorporates online left-eigenvector estimation, but keeps
    the exchanged eigenvector-estimation variables free of communication
    noise. In this paper,
    R-Xi-row considers a more general row-stochastic setting in which
    information-sharing noise simultaneously affects the optimization,
    gradient-tracking, and left-eigenvector-estimation dynamics.
    Noise in the last dynamics introduces an additional difficulty since
    the resulting estimation error directly distorts the gradient scaling
    used for imbalance compensation.
\end{remark}

We make the following assumption for the considered information-sharing noise.

\begin{assumption}[Noise assumption]\label{ass:noise}
${e}_{i,k}^{x}$, $\mathbf{e}_{i,k}^{y}$ and ${e}_{i,k}^{z}$ are mutually independent for all $i\in\mathcal{V}$ and $k\geq 0$. All the noise terms are zero-mean, and their variances are uniformly bounded by a positive constant $\sigma$. Moreover, the random noise $\mathbf{e}_{i,k}^{y}$ satisfies $\|\mathbf{e}_{i,k}^y\|\leq\tilde{\sigma}_y$, $\forall i\in\mathcal{V}$ and $k\geq 0$, almost surely.
\end{assumption}

\section{Convergence analysis}\label{sec:convergence}
In this section, we will show that the sequences generated by the R-Xi-row algorithm converge almost surely (a.s.) to an optimal solution of the problem in \eqref{problem}. 
\subsection{Proof Sketch and Notations}
We establish the almost sure convergence and the expected convergence rate of the proposed R-Xi-row algorithm. The analysis proceeds in four steps. We first characterize the asymptotic behavior of the left eigenvector estimate dynamics and the gradient scaling gain $\kappa_{i,k}$ in Lemma~\ref{y_nr}. We then derive a linear system of conditional inequalities for the error vector $\mathbf{t}_k$ in Lemma~\ref{lemma4}. Based on this recursion and an auxiliary lemma, Theorem~\ref{almost_sure_conv} establishes almost sure convergence. Finally, by establishing a Lyapunov sequence in expectation, Lemma~\ref{thm:random_lyapunov_drift} and Theorem~\ref{cor:polynomial_rate} provide the non-asymptotic convergence rate in expectation.

We define the following notations for convenience:
\begin{eqnarray*}
\tilde{x}_k\hspace{-7pt}&=&\hspace{-7pt}{\mathbf{r}^{\top}\mathbf{x}_k}\in\mathbb{R},\quad\tilde{\mathbf{x}}_k=\mathbf{1}\tilde{x}_k\in\mathbb{R}^n,\\
\tilde{z}_k\hspace{-7pt}&=&\hspace{-7pt}{\mathbf{r}^{\top}\mathbf{z}_k}\in\mathbb{R},\quad\;\tilde{\mathbf{z}}_k=\mathbf{1}\tilde{z}_k\in\mathbb{R}^n,
\end{eqnarray*}
Let $A:=W-I$. Following the transformed-norm construction in
\cite{TAC_GradientTrackingBasedDistributedOptimizationWithGuaranteedOptimalitUnderNoisyInformationSharing},
there exists a fixed nonsingular transformation matrix
$\bar A\in\mathbb{R}^{n\times n}$ such that the vector norm
$\|\cdot\|_A$ is defined by
\[
\|\mathbf{x}\|_A
\triangleq
\|\bar A\mathbf{x}\|_2,
\qquad
\forall\,\mathbf{x}\in\mathbb{R}^n.
\]
For any $M\in\mathbb{R}^{n\times n}$, we use the corresponding
induced matrix norm
\[
\|M\|_A
\triangleq
\sup_{\mathbf{x}\neq 0}
\frac{\|M\mathbf{x}\|_A}{\|\mathbf{x}\|_A}
=
\|\bar A M\bar A^{-1}\|_2.
\]
The matrix $\bar A$ is introduced only for the convergence analysis,
and its explicit expression is not required.
Following the same transformed-norm argument, $\bar A$ can be
selected such that, for some $\rho_A>0$, it has
\[
\|W_k-\mathbf{1}\mathbf r^\top\|_A
\leq 1-\beta_k\rho_A<1.
\]

Using the above norm, we define
\begin{eqnarray*}
\mathbf{t}_k \hspace{-7pt}&=&\hspace{-7pt}
\begin{bmatrix}
t_{1,k}\\
t_{2,k}\\
t_{3,k}
\end{bmatrix}
\overset{\triangle}{=}
\begin{bmatrix}
F(\tilde{x}_k)-F(x^*)\\
\|\mathbf{x}_k-\tilde{\mathbf{x}}_k\|_A^2\\
\|\mathbf{z}_k-\tilde{\mathbf{z}}_k\|_A^2
\end{bmatrix},
\end{eqnarray*}
where $F(\tilde{x}_k)-F(x^\ast)$ measures the optimality gap,
while the remaining terms $\|\mathbf x_k-\tilde{\mathbf{x}}_k\|_A^2$ and
$\|\mathbf z_k-\tilde{\mathbf z}_k\|_A^2$ quantify the consensus errors of the
optimization variable and the gradient-tracking variable,
respectively.

For notational simplicity, we further define
\[
\Pi_r := I-\mathbf{1}\mathbf{r}^\top,\,\,
\delta_k := \max_{i\in\mathcal V}|nr_i\kappa_{i,k}-1|,
\,\,
\kappa_k := \operatorname{diag}\{\kappa_{i,k}\}_{i\in\mathcal V}.
\]

\subsection{Supporting Lemmas}
In what follows, a few auxiliary lemmas are presented to support the convergence analysis.

\begin{lemma}
[Adapted from Lemma~6 in \cite{2021_pushu_TAC_push-pull}]\label{lemma1}
There exists a constant $c_{A,2}>0$ such that, for any
$\mathbf{x}\in\mathbb{R}^n$, it has
\[
\|\mathbf{x}\|_A\leq c_{A,2}\|\mathbf{x}\|_2.
\]
Moreover, after a proper rescaling of the norm $\|\cdot\|_A$, we have
\[
\|\mathbf{x}\|_2\leq\|\mathbf{x}\|_A.
\]
\end{lemma}

\begin{lemma}[Adapted from Theorem 1 in \cite{TAC_GradientTrackingBasedDistributedOptimizationWithGuaranteedOptimalitUnderNoisyInformationSharing}]\label{lemma2}
Assume that the cost function $F(\cdot)$ is continuously differentiable and that the problem is well-defined with optimal solution $x^*$. Suppose that the following relation holds almost surely for some sufficiently large integer $T\geq0$ and for all $k\geq T$,
\begin{align*}
  \mathbb{E}[\mathbf{t}_{k+1}|\mathcal{F}_k] \leq &(G_k+a_k\mathbf{1}\mathbf{1}^{\top})\mathbf{t}_k-H_k
  \begin{bmatrix}
    \|\nabla F(\tilde{x}_k)\|_2^2 \\
    \|\nabla \bar{f}_k\|_2^2
  \end{bmatrix} + b_k\mathbf{1}
\end{align*}
where 
\begin{align*}
    G_k&=
    \begin{bmatrix}
      1 & \eta_1\alpha_k  & 0 \\
      0 & 1-\eta_2\beta_k & \eta_3\beta_k \\
      0 & 0 & 1-\eta_4\beta_k 
    \end{bmatrix},\\
    H_k&=
    \begin{bmatrix}
      \eta_5\alpha_k & \eta_6\alpha_k-\eta_7\alpha_k^2  \\
      0 & 0  \\
      0 & 0  
    \end{bmatrix},
\end{align*}
with $\eta_i>0, \forall 1\leq i\leq 7$ and $\eta_2, \eta_4\in(0,1)$, while the nonnegative scalar sequences $\{a_k\}$, $\{b_k\}$ and positive sequences $\{\alpha_k\}$ and $\{\beta_k\}$ satisfy $\sum_{k=0}^{\infty}a_k<\infty$ a.s., $\sum_{k=0}^{\infty}b_k<\infty$ a.s., $\sum_{k=0}^{\infty}\alpha_k=\infty$, $\sum_{k=0}^{\infty}\beta_k=\infty$, $\sum_{k=0}^{\infty}\beta_k^2<\infty$, $\sum_{k=0}^{\infty}\frac{\alpha_k^2}{\beta_k}<\infty$ and $\lim_{k\rightarrow \infty}\frac{\alpha_k} {\beta_k}=0$. Then, it has
\par a) $\lim_{k\rightarrow \infty} F(\tilde{x}_k)$ exists and $\lim_{k\rightarrow \infty} \|x_{i,k}-\tilde{x}_k\|=\lim_{k\rightarrow \infty} \|z_{i,k}-\tilde{z}_k\|=0$, $\forall i$, a.s.
\par b) $\lim \inf_{k\rightarrow \infty}\|\nabla F(\tilde{x}_k)\|=0$ holds a.s. Moreover,
if the function $F(\cdot)$ has bounded level sets, then $\{\tilde{x}_k\}$ is bounded and every accumulation point of $\{\tilde{x}_k\}$ is an optimal solution and $\lim_{k\rightarrow\infty} F(x_{i,k})=F(x^*), \forall i\in\mathcal{V}$ a.s.
\end{lemma}

\subsection{Almost Sure Convergence}
Next, we proceed to analyze the convergence of R-Xi-row algorithm. The following analysis is carried out on the probability-one event
on which $\|\mathbf e^y_{i,k}\|\leq\tilde{\sigma}_y$
$\forall i\in\mathcal V$ and $k\geq0$.

\begin{lemma}\label{y_nr}
 Let $\Lambda_k=\sum_{t=0}^{k}\lambda_t$ and $\Gamma_k=\sum_{t=0}^{k}\gamma_t$.  Suppose that
$\{\lambda_k\}$ is nonincreasing and the decaying gains
$\lambda_k$ and $\gamma_k$ satisfy $\sum_{k=0}^{\infty}\lambda_k=\infty$, $\sum_{k=0}^{\infty}\gamma_k=\infty$, $\lim_{k\rightarrow \infty} \frac{1}{{\Gamma_k}} \sup_{0\leq l\leq k}\frac{\gamma_l}{\lambda_l}= 0$, $\lim_{k\rightarrow\infty}\frac{\Lambda_k}{\Gamma_k}=0$ and $\lambda_k\in(0,1]$ for all $k\geq 0$. Under Assumptions~ \ref{ass:network} and \ref{ass:noise}, the gradient scaling gain
 $\kappa_{i,k}$ defined in \eqref{eq:safeguarded_kappa} asymptotically converges almost surely to $\frac{1}{nr_i}$, i.e.,
 $\lim_{k\to\infty}\kappa_{i,k}=\frac{1}{nr_i}$, $\forall i\in\mathcal V$ a.s.
\end{lemma}

\begin{proof}
    See Appendix~\ref{app:proof-y-nr}.
\end{proof}

\begin{remark}[Role of the normalized scaling gain]\label{remark2}
Since both $\lambda_k\mathbf e_k^y$ and $\gamma_kY_0$ accumulate in
the eigenvector-estimation recursion, the magnitude of
$[\mathbf y_{i,k}]_i$ grows with the accumulated correction.
Hence, directly using its reciprocal for gradient scaling would lead
to an excessively small step size. Instead, when the
safeguard in \eqref{eq:safeguarded_kappa} is inactive, the normalized
ratio $
\frac{\sum_{j=1}^{n}[\mathbf y_{i,k+1}]_j}
{n[\mathbf y_{i,k+1}]_i}$
is used. By Lemma~\ref{y_nr}, the safeguard is eventually inactive
almost surely, and $
\lim_{k\to\infty}\kappa_{i,k}
=
\frac{1}{nr_i}$ a.s.
Thus, the normalization cancels the common growth of the eigenvector
estimate while recovering the expected gradient scaling by $r_i$.
\end{remark}

With the above preparations, we start to establish a linear system of inequalities for the relation between $\mathbf{t}_{k+1}$ and $\mathbf{t}_k$.

\begin{lemma}\label{lemma4}
Under Assumptions~\ref{ass:network}--\ref{ass:noise}, the sequences generated by R-Xi-row satisfy
\begin{eqnarray}\label{lemma_64}
   \hspace{-27pt}&&\hspace{-7pt}  \mathbb{E}\left[\mathbf{t}_{k+1}|\mathcal{F}_k\right] \leq (G_k+A_k)\mathbf{t}_k-H_k
   \begin{bmatrix}
     \|\nabla F(\tilde{x}_k)\|_2^2 \\
     \|\nabla \bar{f}_k\|_2^2
   \end{bmatrix}+B_k,
 \end{eqnarray}
where $\mathcal{F}_k=\{x_{i,l},\mathbf{y}_{i,l},z_{i,l}:0\leq l\leq k, i\in\mathcal{V}\}$ and $\nabla \bar{f}_k=\frac{1}{n}\sum_{i=1}^{n}\nabla f_i(x_{i,k})$. The elements of the system matrices $G_k$, $A_k$, $H_k$ and $B_k$ are given by
\begin{eqnarray}
   \hspace{-13pt}   A_k  \hspace{-7pt}&=&\hspace{-7pt}
  \begin{bmatrix}\label{38_2}
     8\alpha_kL\delta_k^2 & \frac{2\alpha_kL^2\delta_k^2}{n} & 0\\
     \frac{16nLc_{A,2}^2\alpha_k^2}{\beta_k\rho_A}\|\Pi_r\kappa_k\|_{{A}}^2 & \frac{4\alpha_k^2 L^2c_{{A},2}^2}{\beta_k\rho_A}\|\Pi_r\kappa_k\|_{{A}}^2 & 0\\
     \frac{8nL\alpha_k^2 c_{A,2}^2}{\beta_k\rho_A}\|\Pi_r\kappa_k\|_{{A}}^2 & \frac{2\alpha_k^2 L^2 c_{{A},2}^2 }{\beta_k\rho_A}\|\Pi_r\kappa_k\|_{{A}}^2 & 0
    \end{bmatrix},       \\
    \hspace{-13pt}    G_k  \hspace{-7pt}&=&\hspace{-7pt}
    \begin{bmatrix}\label{37_2}
      1 & \frac{\alpha_k L^2}{n} & 0\\
      0 & 1-\beta_k\rho_A & \frac{2\beta_k}{\rho_A}\|\Pi_r A\|_{{A}}^2 \\
      0 & 0 & 1-\beta_k\rho_A 
    \end{bmatrix},\\
\hspace{-13pt}H_k\hspace{-7pt}&=&\hspace{-7pt}\begin{bmatrix}\label{39_2}
                                  \frac{\alpha_k}{2} & \frac{\alpha_k-L\alpha_k^2}{2} \\
                                  0 & 0 \\
                                  0 & 0 
                                \end{bmatrix},
                                \\
  \hspace{-13pt}  B_k
       \hspace{-7pt} &=&\hspace{-7pt}
      \begin{bmatrix}\label{40_2}
        \frac{\alpha_k\delta_k^2}{n}\|\nabla \mathbf{f}^*\|_2^2 +  L\sigma\beta_k^2\|\mathbf{r}\|_2^2\\
         \frac{8\alpha_k^2 c_{A,2}^2}{\beta_k\rho_A}\|\Pi_r \kappa_k\|_{{A}}^2 \|\nabla \mathbf{f}^*\|_2^2 +2\beta_k^2c_{A,2}^2\|\Pi_r\|_{{A}}^2 n\sigma\\
         \frac{4\alpha_k^2 c_{A,2}^2}{\beta_k\rho_A}\|\Pi_r\kappa_k\|_{{A}}^2 \|\nabla \mathbf{f}^*\|_2^2 +\beta_k^2\|\Pi_r\|_{{A}}^2 c_{{A},2}^2 n\sigma
        \end{bmatrix}.
\end{eqnarray}
\end{lemma}

\begin{proof}
    The proof is provided in Appendix~\ref{app:proof-lemma6}.
\end{proof}

Here, we are ready to show the convergence of R-Xi-row.
\begin{theorem}[Almost sure convergence]\label{almost_sure_conv}
  Let Assumptions~\ref{ass:network}--\ref{ass:noise} hold. Suppose the decaying gains $\lambda_k$ and $\gamma_k$ satisfy the conditions in Lemma~\ref{y_nr}. Suppose further that $\alpha_k$ and $\beta_k$ satisfy $\sum_{k=0}^{\infty}\alpha_k=\infty$, $\sum_{k=0}^{\infty}\beta_k=\infty$,  
  $\sum_{k=0}^{\infty}\beta_k^2<\infty$, $\sum_{k=0}^{\infty}\frac{\alpha_k^2}{\beta_k}<\infty$, $\lim_{k\rightarrow \infty}\frac{\alpha_k}{\beta_k}=0$ with $\beta_k\in(0,1]$ for all $k$, and the coupling condition $\sum_{k=0}^{\infty}\alpha_k\left(\frac{1+\sup_{0\leq l\leq k}\gamma_l/\lambda_l}{\Gamma_k}\right)^2<\infty$
  holds. Then, for every $i\in\mathcal V$, the following limits hold almost surely:
    \begin{align*}
    \lim_{k\to\infty}\|x_{i,k}-\tilde{x}_k\|
    &=
    \lim_{k\to\infty}\|z_{i,k}-\tilde{z}_k\|
    =0,\\
    \lim_{k\to\infty}F(x_{i,k})
    &=F(x^*).
    \end{align*}
\end{theorem}

\begin{proof}
By Lemma~\ref{y_nr}, for any $\varepsilon>0$, there exists a finite integer $T(\varepsilon)$ such that $\kappa_{i,k}\in\left[ \frac{1-\varepsilon}{nr_i},\frac{1+\varepsilon}{nr_i} \right]$ a.s. for all $i\in\mathcal{V}$ and $k\geq T(\varepsilon)$. Hence, there exists a finite constant $C>0$ such that $\|\Pi_r\kappa_k\|_A\leq C$ a.s. Notice that $\sum_{k=0}^{\infty}\beta_k^2<\infty$ and $\sum_{k=0}^{\infty}\frac{\alpha_k^2}{\beta_k}<\infty$, that is, all the terms in $A_k$ and $B_k$ except $\frac{\alpha_k\delta_k^2}{n}\|\nabla \mathbf{f}_k\|_2^2$ are summable. According to Lemma~\ref{lemma2}, for a sequence in the form like \eqref{lemma_64}, it remains to verify if $\sum_{k=0}^{\infty}\alpha_k\delta_k^2<\infty$ is summable under the conditions of decaying gains in the statement of the theorem.

Recalling that $\delta_k:=\max_{i\in\mathcal{V}}|nr_i\kappa_{i,k}-1|$, we next start to analyze 
\(|nr_i\kappa_{i,k}-1|\). Let \(\hat{\mathbf{y}}_{i,k+1}\in\mathbb{R}^n\) denote the column vector formed by the entries of the \(i\)-th row of \(\hat{Y}_{k+1}\) where $\hat{Y}_k:=(I-\mathbf{1}\mathbf{r}^{\top})Y_k$ is given in the proof of Lemma~\ref{y_nr}. With a slight abuse of notation, $e_i$ is used to denote the vector $[0,...,1,...,0]^{\top}$ where the only nonzero entry is the $i$-th one. 

By Lemma~\ref{y_nr}, on the probability-one event, there exists a sufficiently large integer \(T\ge 0\) and a constant \(c_1\in(0,1)\) such that, for all \(k\ge T\) and all \(i\in\mathcal{V}\), it has $[\mathbf{r}^\top Y_{k+1}]_i \ge c_1\Gamma_k r_i$ and $|[\hat{\mathbf{y}}_{i,k+1}]_i|
\le \frac{c_1}{2}\Gamma_k r_i$. 
Hence, for all $k\ge T$, we derive 
\begin{equation}
    e_i^\top \mathbf{y}_{i,k+1}=[\mathbf{r}^\top Y_{k+1}]_i+[\hat{\mathbf{y}}_{i,k+1}]_i
    \ge \frac{c_1}{2}\Gamma_k r_i >0. 
\end{equation}
Since $\Gamma_k\rightarrow\infty$ and $r_i>0$ for every
$i\in\mathcal V$, by further enlarging $T$ if necessary, we have $\frac{c_1}{2}\Gamma_k r_i>\varepsilon_{\kappa}$,
$\forall i\in\mathcal V, k\geq T$.
Therefore, the safeguard in \eqref{eq:safeguarded_kappa} is
inactive for all $k\geq T$, and $\kappa_{i,k}$ is given by the first branch of \eqref{eq:safeguarded_kappa}. Meanwhile, there exists a constant \(c_2>0\) such that \begin{equation}
  \big|r_i\mathbf{1}^\top \hat{\mathbf{y}}_{i,k+1}-e_i^\top \hat{\mathbf{y}}_{i,k+1}\big|
\le c_2\|\hat{\mathbf{y}}_{i,k+1}\|_\mathbf{r}
\le c_2\|\hat{Y}_{k+1}\|_\mathbf{r}.
\end{equation}
Therefore, for all \(k\ge T\), one can obtain
\begin{align}\label{25_38}
|nr_i\kappa_{i,k}-1|
&=
\left|
\frac{r_i\mathbf{1}^\top \mathbf{y}_{i,k+1}-e_i^\top \mathbf{y}_{i,k+1}}
{e_i^\top \mathbf{y}_{i,k+1}}
\right|
\nonumber\\
&=
\left|
\frac{r_i\mathbf{1}^\top \hat{\mathbf{y}}_{i,k+1}-e_i^\top \hat{\mathbf{y}}_{i,k+1}}
{e_i^\top \mathbf{y}_{i,k+1}}
\right|\le
\frac{2c_2}{c_1r_i}
\frac{\|\hat{Y}_{k+1}\|_\mathbf{r}}{\Gamma_k}.
\end{align}
Taking the maximum over \(i\in\mathcal{V}\) and substituting \eqref{new_1} into \eqref{25_38}, we have
\begin{equation}\label{eq:delta_bound}
  \delta_k
\le
\frac{C_0}{\Gamma_k}
+
\frac{C_1}{\Gamma_k}\sup_{0\le l\le k}\frac{\gamma_l}{\lambda_l},
\qquad \forall k\ge T,
\end{equation}
where $C_0$ and $C_1$ are two positive constants.

Thus, under the condition $\sum_{k=0}^{\infty}\alpha_k\left(\frac{1+\sup_{0\leq l\leq k}\gamma_l/\lambda_l}{\Gamma_k}\right)^2<\infty$, we have that $\sum_{k=0}^{\infty}\alpha_k\delta_k^2<\infty$ a.s. Based on the above analysis, it is observed that the decaying gains given in Theorem~\ref{almost_sure_conv} satisfy the statement as required in Lemma~\ref{lemma2}. Therefore, all the summability conditions required by
Lemma~\ref{lemma2} are satisfied almost surely. Applying
Lemma~\ref{lemma2} to \eqref{lemma_64} yields
$
\lim_{k\to\infty}\|x_{i,k}-\tilde{x}_k\|
=
\lim_{k\to\infty}\|z_{i,k}-\tilde{z}_k\|
=0$,
$\lim_{k\to\infty}F(x_{i,k})=F(x^*)$ $\forall i\in\mathcal V$ a.s.
This completes the proof.
\end{proof}

By Lemma~\ref{y_nr} and Theorem~\ref{almost_sure_conv}, the roles of different decaying gains in R-Xi-row can be interpreted as follows: (i) In \eqref{re_3a}, $\lambda_k$ regulates the consensus in the distributed estimation of the left
eigenvector, while $\gamma_k$ controls the correction by $Y_0$.
The conditions
$\sum_{k=0}^{\infty}\lambda_k=\infty$ and
$\sum_{k=0}^{\infty}\gamma_k=\infty$
ensure persistent excitation of these two processes.
Premultiplying \eqref{recursive_Y} by $\mathbf r^\top$
shows that the accumulated noise in $\mathbf r^\top Y_{k+1}$ is
$\sum_{\ell=0}^{k}\lambda_\ell\mathbf r^\top\mathbf e_\ell^y$.
Since $\mathbf e_\ell^y$ is uniformly bounded,
$\Lambda_k/\Gamma_k\to0$ makes this term negligible after
normalization by $\Gamma_k$, leading to \eqref{25_1}. Meanwhile,
\eqref{new_1} shows that
$\Gamma_k^{-1}\sup_{0\le\ell\le k}\gamma_\ell/\lambda_\ell\to0$
makes $\|\hat Y_{k+1}\|_{\mathbf r}/\Gamma_k$ vanish, as in
\eqref{24_1}. Together, these results give
$Y_k/\Gamma_k\to\mathbf1\mathbf r^\top$ and hence
$\kappa_{i,k}\to\frac{1}{nr_i}$;
(ii) $\alpha_k$ and $\beta_k$ control the optimization and
consensus or tracking dynamics in
\eqref{re_3b}--\eqref{re_3c}. In the proof of
Lemma~\ref{thm:random_lyapunov_drift}, the consensus and tracking
contraction is of order $\beta_k$, whereas the optimization
dynamics introduces coupling terms of orders $\alpha_k$ and
$\frac{\alpha_k^2}{\beta_k}$. Thus, $\frac{\alpha_k}{\beta_k}\to0$ makes these
terms asymptotically small relative to the network contraction,
giving the required time-scale separation. The nonsummability of
$\alpha_k$ and $\beta_k$ keeps both dynamics active,
$\sum_{k=0}^{\infty}\beta_k^2<\infty$ limits the accumulation of information-sharing noise, and
$\sum_{k=0}^{\infty}\frac{\alpha_k^2}{\beta_k}<\infty$ makes the
optimization-induced coupling summable;
(iii) The condition
$
\sum_{k=0}^{\infty}\alpha_k
\left(
\frac{1+\sup_{0\le\ell\le k}\gamma_\ell/\lambda_\ell}
{\Gamma_k}
\right)^2<\infty
$
limits the cumulative bias caused by using inexact gradient scaling.
Since \eqref{eq:delta_bound} implies
$
\delta_k=
O\!\left(
\frac{1+\sup_{0\le\ell\le k}\gamma_\ell/\lambda_\ell}
{\Gamma_k}
\right),
$
this condition gives 
$\sum_{k=0}^{\infty}\alpha_k\delta_k^2<\infty$,
therewith the transient gradient-scaling error is suppressed.

\begin{remark}[Practical selection of decaying gains]\label{remark3}
Specific settings of the decaying gains in the statement of Theorem~\ref{almost_sure_conv} can be the following polynomially decaying gains: 
$\lambda_k=\frac{c_\lambda}{(k+1)^p}$, 
$\gamma_k=\frac{c_\gamma}{(k+1)^q}$, 
$\alpha_k=\frac{c_\alpha}{(k+1)^a}$ and 
$\beta_k=\frac{c_\beta}{(k+1)^b}$
where $c_\lambda,c_\beta\in(0,1]$ and $c_\gamma,c_\alpha>0$ are constants, and the exponents satisfy $0<q<p<1,\ \frac12<b<a<1,\ a>2p-1$ and $2a-b>1$. 
For the exponent conditions of these polynomial gains, 
$q<p$ gives $\Lambda_k/\Gamma_k\to0$, while $p<1$ gives
$\Gamma_k^{-1}\sup_{0\le\ell\le k}\gamma_\ell/\lambda_\ell\to0$.
For $\alpha_k$ and $\beta_k$, the requirements $b<a$, $b>\frac{1}{2}$ and $2a-b>1$ imply
$\frac{\alpha_k}{\beta_k}\to0$, 
$\sum_{k=0}^{\infty}\beta_k^2<\infty$ and
$\sum_{k=0}^{\infty}\frac{\alpha_k^2}{\beta_k}<\infty$, respectively. Moreover,
\eqref{eq:delta_bound} gives
$\delta_k=O((k+1)^{-(1-p)})$, hence $a>2p-1$ ensures
$\sum_{k=0}^{\infty}\alpha_k\delta_k^2<\infty$. Therefore, in practical tuning, one may first choose
$0<q<p<1$, then choose $b>\frac{1}{2}$, and finally let
$a\in(b,1)$ such that $a>2p-1$ and $2a-b>1$. The coefficients $c_\lambda$ and
$c_\beta$ are selected in $(0,1]$ as required by
Lemma~\ref{y_nr} and Theorem~\ref{almost_sure_conv}, while
$c_\gamma$ and $c_\alpha$ can be adjusted to tune the transient
behavior.
\end{remark}

\subsection{Expected Convergence Rate}
Theorem~\ref{almost_sure_conv} establishes the almost sure convergence of R-Xi-row in the asymptotic regime. We next investigate its convergence rate in expectation. To this end, we construct a Lyapunov sequence and derive a deterministic recursion for its expectation. For notational convenience, define
\begin{equation}\label{delta_bar_define}
  \bar{\delta}_k:=\frac{C_0}{\Gamma_k}+
\frac{C_1}{\Gamma_k}\sup_{0\le l\le k}\frac{\gamma_l}{\lambda_l}
\end{equation}
which coincides with the right-hand side of \eqref{eq:delta_bound}. By Lemma~\ref{y_nr} and the bound in \eqref{eq:delta_bound}, there exists a sufficiently large integer \(T\ge 0\) and a constant \(M_\kappa>0\) such that for all \(k\ge T\), one has $\delta_k \le \bar\delta_k$ and $\|\Pi_r \kappa_k\|_A \le M_\kappa$ almost surely.

\begin{lemma}
\label{thm:random_lyapunov_drift}
Define the random Lyapunov sequence
\[
\Phi_k := t_{1,k} + \omega_1 t_{2,k} + \omega_2 t_{3,k},
\]
where $\omega_1>0$ and $\omega_2 > \frac{4\omega_1}{\rho_A^2}\|\Pi_r A\|_A^2$. Suppose Assumptions~\ref{ass:network}--\ref{ass:noise} hold, the sequences $\{\lambda_k\}$ and $\{\gamma_k\}$ satisfy the conditions in Lemma~\ref{y_nr}, and $\{\alpha_k\}$ and $\{\beta_k\}$ satisfy $\lim_{k\to\infty}\frac{\alpha_k}{\beta_k}=0$. Then, for any fixed $\eta\in(0,\mu)$, there exist constants $C_1,C_2,C_3>0$ and a sufficiently large integer $T\geq 0$ such that, for all $k\geq T$,
\begin{equation}\label{phir}
  \mathbb E[\Phi_{k+1}| \mathcal F_k]
\le
(1-\eta\alpha_k)\Phi_k
+C_1\alpha_k\bar\delta_k^2
+C_2\frac{\alpha_k^2}{\beta_k}
+C_3\beta_k^2.
\end{equation}
\end{lemma}

\begin{proof}
By Lemma~\ref{y_nr} and \eqref{eq:delta_bound}, there exist a sufficiently large integer $T\geq 0$ and a constant $M_\kappa>0$ such that $\delta_k\leq \bar{\delta}_k$ and $\|\Pi_r\kappa_k\|_A\leq M_\kappa$ for all $k\geq T$ almost surely. Using these bounds together with the $\mu$-strong convexity of $F$, the recursions in \eqref{F_F^*}, \eqref{s2_5} and \eqref{theorem2_eq17} can be rewritten for all $k\geq T$ as follows:
\begin{subequations}\label{eq:short_rec}
\begin{align}
\label{eq:short_rec_E}
\mathbb E[t_{1,k+1}| \mathcal F_k]
&\le
\bigl(1-\alpha_k\mu+8L\alpha_k\bar\delta_k^2\bigr)t_{1,k} \nonumber\\
&\quad
+\frac{\alpha_kL^2+2\alpha_kL^2\bar\delta_k^2}{n}t_{2,k}
+b_k^{(1)},
\\
\label{eq:short_rec_X}
\mathbb E[t_{2,k+1}| \mathcal F_k]
&\le
\left(1-\beta_k\rho_A+c_1\frac{\alpha_k^2}{\beta_k}\right)t_{2,k}
+c_2\beta_k t_{3,k} \nonumber\\
&\quad
+c_3\frac{\alpha_k^2}{\beta_k}t_{1,k}
+b_k^{(2)},
\\[1mm]
\label{eq:short_rec_Z}
\mathbb E[t_{3,k+1}| \mathcal F_k]
&\le
(1-\beta_k\rho_A)t_{3,k}
+c_4\frac{\alpha_k^2}{\beta_k}t_{2,k}
+c_5\frac{\alpha_k^2}{\beta_k}t_{1,k}
+b_k^{(3)}.
\end{align}
\end{subequations}
where for brevity, let 
$c_1:=\frac{4L^2c_{A,2}^2M_\kappa^2}{\rho_A}$, 
$c_2:=\frac{2}{\rho_A}\|\Pi_rA\|_A^2$, $c_3:=\frac{16nLc_{A,2}^2M_\kappa^2}{\rho_A}$, 
$c_4:=\frac{2L^2c_{A,2}^2M_\kappa^2}{\rho_A}$, $c_5:=\frac{8nLc_{A,2}^2M_\kappa^2}{\rho_A}$, 
$b_k^{(1)}=
\frac{\alpha_k\bar\delta_k^2}{n}\|\nabla \mathbf f^*\|_2^2
+L\sigma\beta_k^2\|\mathbf r\|_2^2$, 
$b_k^{(2)}
=
\frac{8\alpha_k^2 c_{A,2}^2M_\kappa^2}{\beta_k\rho_A}\|\nabla \mathbf f^*\|_2^2
+2\beta_k^2 c_{A,2}^2\|\Pi_r\|_A^2 n\sigma$, and 
$b_k^{(3)}=
\frac{4\alpha_k^2 c_{A,2}^2M_\kappa^2}{\beta_k\rho_A}\|\nabla \mathbf f^*\|_2^2
+\beta_k^2 c_{A,2}^2\|\Pi_r\|_A^2 n\sigma$.

Multiplying \eqref{eq:short_rec_X} and \eqref{eq:short_rec_Z} by $\omega_1$ and $\omega_2$, respectively, and adding the resulting inequalities to \eqref{eq:short_rec_E}, we arrive at the following one-step drift inequality for the Lyapunov sequence:
\begin{equation}\label{phi_concise}
\mathbb E[\Phi_{k+1} | \mathcal F_k]
\le
\Theta_{1,k}t_{1,k}+\Theta_{2,k}t_{2,k}+\Theta_{3,k}t_{3,k}+\bar b_k,
\end{equation}
where $
\Theta_{1,k}
=
1-\alpha_k\mu+8L\alpha_k\bar\delta_k^2
+(\omega_1c_3+\omega_2c_5)\frac{\alpha_k^2}{\beta_k}$, 
$
\Theta_{2,k}
=
\frac{\alpha_kL^2+2\alpha_kL^2\bar\delta_k^2}{n}
+\omega_1\left(1-\beta_k\rho_A+c_1\frac{\alpha_k^2}{\beta_k}\right)
+\omega_2c_4\frac{\alpha_k^2}{\beta_k}
$, 
$
\Theta_{3,k}
=
\omega_1c_2\beta_k+\omega_2(1-\beta_k\rho_A)
$, 
and 
$
\bar b_k:=b_k^{(1)}+\omega_1b_k^{(2)}+\omega_2b_k^{(3)}
$.

Fix any $\eta\in(0,\mu)$. By enlarging $T$, one can ensure that
$\Theta_{1,k}\le 1-\eta\alpha_k$,
$\Theta_{2,k}\le \omega_1(1-\eta\alpha_k)$ and 
$\Theta_{3,k}\le \omega_2(1-\eta\alpha_k)$ for all $k\geq T$.
Substituting these bounds into \eqref{phi_concise} gives
\begin{equation}\label{E_phi}
  \mathbb E[\Phi_{k+1} | \mathcal F_k]
\le
(1-\eta\alpha_k)\Phi_k+\bar b_k.
\end{equation}
Since $\bar{b}_k$ can be bounded by $\bar{b}_k \leq C_1\alpha_k\bar{\delta}_k^2 + C_2\frac{\alpha_k^2}{\beta_k}+C_3\beta_k^2$ for some positive constants $C_1,C_2,C_3$, the desired recursion in \eqref{phir} follows. 
This completes the proof.
\end{proof}

By virtue of the above Lyapunov sequence, a deterministic expected recursion denoted by $\bar{\Phi}_k:=\mathbb E[\Phi_k]$ is obtained. Taking total expectation on both sides of the drift inequality in Lemma~\ref{thm:random_lyapunov_drift} and using the tower property yield
\begin{eqnarray}
\hspace{-15pt}\mathbb E[\Phi_{k+1}]
\hspace{-7pt}&=&\hspace{-7pt}
\mathbb E\!\left[\mathbb E[\Phi_{k+1}| \mathcal F_k]\right]\nonumber\\
\hspace{-7pt}&\le&\hspace{-7pt}
(1-\eta\alpha_k)\mathbb E[\Phi_k]
+
C_1\alpha_k\bar\delta_k^2
+
C_2\frac{\alpha_k^2}{\beta_k}
+
C_3\beta_k^2.
\end{eqnarray}
Then, for all $k\ge T$, we have
\begin{equation}\label{phi^'}
  \bar{\Phi}_{k+1}
\le
(1-\eta\alpha_k)\bar{\Phi}_k
+
C_1\alpha_k\bar\delta_k^2
+
C_2\frac{\alpha_k^2}{\beta_k}
+
C_3\beta_k^2.
\end{equation}

We are now in a position to derive an explicit convergence rate under the polynomial decaying gains in Remark~\ref{remark3}. To this end, we first recall the following auxiliary lemma.

\begin{lemma}
\label{lem:scalar_recursion_estimate}
Let $\{\psi_k\}$ be a nonnegative sequence such that, for all sufficiently large $k$, $
\psi_{k+1}\le
\left(1-\frac{\vartheta}{(k+1)^a}\right)\psi_k
+\frac{C}{(k+1)^s}$,
where $\vartheta>0$, $C>0$, $0<a<1$ and $s>a$. Then, we have
\begin{eqnarray}
\psi_k\hspace{-7pt}&=&\hspace{-7pt}
\mathcal{O}\!\left((k+1)^{-(s-a)}\right).
\end{eqnarray}
\end{lemma}

\begin{proof}
See Appendix~\ref{app:proof-scalar}.
\end{proof}

\begin{theorem}[Sublinear expected convergence rate]
\label{cor:polynomial_rate}
Let assumptions in Lemma~\ref{thm:random_lyapunov_drift} hold, and the decaying gains be chosen according to Remark~\ref{remark3}. Then the optimality gap and the consensus/tracking errors converge in expectation at the rate $\mathcal{O}\!\left((k+1)^{-m}\right)$ where $m=\min\{2(1-p),\,a-b,\,2b-a\}$.
\end{theorem}

\begin{proof}
Under the polynomial decaying gains in Remark~\ref{remark3}, one has $\Gamma_k=\mathcal{O}((k+1)^{1-q})$ and $ \sup_{0\leq \ell\leq k}\frac{\gamma_\ell}{\lambda_\ell} =\mathcal{O}((k+1)^{p-q})$, and hence, by \eqref{delta_bar_define}, we have $\bar{\delta}_k = \mathcal{O}((k+1)^{-(1-q)})+\mathcal{O}((k+1)^{-(1-p)}) = \mathcal{O}((k+1)^{-(1-p)})$ where the last equality follows from $q<p$. 

It follows that 
$
\alpha_k\bar\delta_k^2=
\mathcal{O}\!\left((k+1)^{-(a+2-2p)}\right)$, 
$
\frac{\alpha_k^2}{\beta_k}=
\mathcal{O}\!\left((k+1)^{-(2a-b)}\right)$ 
and 
$
\beta_k^2=
\mathcal{O}\!\left((k+1)^{-2b}\right)$. 
In light of these relations and \eqref{phi^'}, we have that there exists a constant \(C>0\) such that
\begin{eqnarray}
\bar{\Phi}_{k+1}\hspace{-7pt}&\le&\hspace{-7pt}
\left(1-\frac{\eta c_\alpha}{(k+1)^a}\right)\bar{\Phi}_k
+\frac{C}{(k+1)^s},
\end{eqnarray}
for all sufficiently large \(k\), where 
$s=
\min\{a+2-2p,2a-b,2b\}$. 
Since 
$
s-a=
\min\{2(1-p),\,a-b,\,2b-a\}$ 
and the exponent conditions in Remark~\ref{remark3} imply \(s>a\), we have the following relation by virtue of Lemma~\ref{lem:scalar_recursion_estimate}: 
\begin{eqnarray}
\bar{\Phi}_k\hspace{-7pt}&=&\hspace{-7pt}
\mathcal{O}\!\left((k+1)^{-(s-a)}\right).
\end{eqnarray}

Recalling the definition of \(\Phi_k\), i.e., $\Phi_k=
t_{1,k}+\omega_1 t_{2,k}+\omega_2 t_{3,k}$, one has that
\begin{eqnarray}
t_{1,k}\hspace{-7pt}&\le&\hspace{-7pt}
\Phi_k,\quad
t_{2,k}\le \frac{\Phi_k}{\omega_1},\quad
t_{3,k}\le \frac{\Phi_k}{\omega_2}.
\end{eqnarray}
Taking expectations over each element gives rise to
\begin{eqnarray}
\mathbb E[t_{1,k}]\hspace{-7pt}&\le&\hspace{-7pt}
\bar{\Phi}_k,\quad
\mathbb E[t_{2,k}]\le \frac{\bar{\Phi}_k}{\omega_1},\quad
\mathbb E[t_{3,k}]\le \frac{\bar{\Phi}_k}{\omega_2}.
\end{eqnarray}

Thus, the proof is completed.
\end{proof}

Based on Theorem~\ref{cor:polynomial_rate}, we next characterize the
supremum of the achievable rate exponent $m$ under the specific polynomial
gains in Remark~\ref{remark3}.
\begin{corollary}[Near-$O(k^{-\frac{1}{3}})$ expected convergence rate]
Under the conditions of Theorem~\ref{cor:polynomial_rate}, for any sufficiently small
$\epsilon>0$, the decaying gains can be chosen such that
\[
\mathbb E[t_{j,k}]
=
O\!\left((k+1)^{-\frac13+\epsilon}\right),
\qquad j=1,2,3.
\]
\end{corollary}

\begin{proof}
Since $m\le\min\{a-b,2b-a\}$, the largest value of the
right-hand side for a fixed $a$ is obtained by balancing the two
terms, i.e., $a-b=2b-a$, which gives $b=\frac{2a}{3}$ and
$m\le\frac{a}{3}$. Recalling that the rate analysis in Lemma~\ref{lem:scalar_recursion_estimate}
requires $a<1$, it follows that $m<\frac{1}{3}$.
On the other hand, this upper bound can be approached arbitrarily
closely. For example, for any $\epsilon\in(0,\frac{1}{12})$,
choosing $a=1-3\epsilon$, $b=\frac{2}{3}-2\epsilon$,
$p=\frac{3}{4}$ together with any
$q\in(0,\frac{3}{4})$ satisfies the conditions in
Remark~\ref{remark3} and yields $m=\frac{1}{3}-\epsilon$. Therefore, 
$\sup m=\frac{1}{3}$.
\end{proof}

\begin{remark}[Unbounded eigenvector-estimation noise]
Assumption~\ref{ass:noise} imposes a uniform boundedness condition on
$\mathbf e^y_{i,k}$, i.e., $\|\mathbf{e}_{i,k}^y\|\leq\tilde{\sigma}_y$ $\forall i\in\mathcal{V}$ and $k\geq 0$. This condition can be relaxed to
unbounded noise in the almost sure convergence analysis. For example,
suppose that $\mathbf e^y_{i,k}$ has uniformly sub-Gaussian tails.
By the Borel--Cantelli lemma,
$\max_{i\in\mathcal V}\|\mathbf e^y_{i,k}\|
=O(\sqrt{\log(k+2)})$ almost surely.
Accordingly, the argument in Lemma~\ref{y_nr} can be retained by strengthening
$\frac{\Lambda_k}{\Gamma_k}\to0$ to
$\frac{\Lambda_k\sqrt{\log(k+2)}}{\Gamma_k}\to0$,
while keeping
$\frac{1}{\Gamma_k}\sup_{0\le\ell\le k}
\frac{\gamma_\ell}{\lambda_\ell}\to0$.
For Theorem~\ref{almost_sure_conv}, the additional noise term can be handled by further
requiring
$\sum_{k=0}^{\infty}
\frac{\alpha_k\log(k+2)}{\Gamma_k^2}<\infty$.
For the polynomial gains in Remark~\ref{remark3}, these strengthened conditions
are also satisfied by the stated exponent conditions.
However, extending the convergence rate result in Theorem~\ref{cor:polynomial_rate} requires
additional analysis. As seen from the conditional recursions prior to the deterministic bounds used in Lemma~\ref{thm:random_lyapunov_drift}, terms involving products such as
$\delta_k^2 t_{j,k}$ and
$\|\Pi_r\kappa_k\|_A^2 t_{j,k}$ arise.
Under sub-Gaussian noise, these bounds are no longer deterministic
and depend on the realized noise sequence. Consequently, such products
require additional probabilistic bounds, such as higher-order moment
or high-probability tail bounds, before the expectation recursion can
be closed.
\end{remark}

\section{Numerical Experiments}\label{sec:numerical}
We evaluate the proposed algorithm on a distributed logistic-regression problem over a ring graph with $10$ agents. The local objective function of agent $i$ is chosen as
\[
f_i(x)=\frac{a}{2n}\|x\|_2^2+\sum_{j=1}^{m_i}\ln\!\Big(1+\exp\!\big(-c_{ij} b_{ij}^{\top}x\big)\Big),
\]
where $n$ is the number of agents, $m_i$ is the number of local training samples at agent $i$, $b_{ij}$ is the feature vector of the $j$th sample, and $c_{ij}$ is its associated label.

In all experiments, the information-sharing noise terms in the updates of $\mathbf{x}_k$ and $\mathbf{z}_k$ are generated as i.i.d. Gaussian random variables $\mathbf{e}_k^x,\mathbf{e}_k^z\sim \mathcal{N}(0,5I)$. 
For the update of $\mathbf{y}_{i,k}$, we use clipped Gaussian noise so that $\|e_{i,k}^y\|\le \tilde\sigma_y=3$. The decaying gains in R-Xi-row algorithm are chosen as 
$
\alpha_k=\frac{0.3}{1+0.2k^{0.85}}, 
\beta_k=\frac{1}{1+0.2k^{0.6}}, 
\lambda_k=\frac{0.6}{1+0.2k^{0.85}}, 
\gamma_k=\frac{1.2}{1+0.2k^{0.8}}$. 
As baselines, we consider Algorithm in \cite{TAC_GradientTrackingBasedDistributedOptimizationWithGuaranteedOptimalitUnderNoisyInformationSharing}, Xi-row~\cite{Linear_convergence_in_optimization_over_directed_graphs_with_row-stochastic_matrices}, and FRSD~\cite{TAC_2024_Fast_row_stochastic}. Algorithm in \cite{TAC_GradientTrackingBasedDistributedOptimizationWithGuaranteedOptimalitUnderNoisyInformationSharing} uses the same decaying step sizes $\alpha_k$ and $\beta_k$ as R-Xi-row. Xi-row and FRSD use the constant step size $\alpha=0.01$ in their original linear-convergence designs, while FRSD additionally uses $\beta=5$. To examine their robustness under noisy information sharing, we also equip Xi-row and FRSD with a diminishing-mixing factor 
$\beta_k=\frac{1}{1+0.2k^{0.6}}$. 
All reported results are averaged over $500$ Monte Carlo runs.

We first consider the scenario in which information-sharing noise terms $\mathbf{e}_k^x$ and $\mathbf{e}_k^z$ appear, while the left eigenvector estimate dynamics remains noise-free. In this case, we refer to the resulting variant of R-Xi-row as \emph{Simplified R-Xi-row}. For each agent $i\in\mathcal V$, \eqref{re_3b} becomes $
z_{i,k+1}=(1-\beta_k)z_{i,k}+\beta_k\hat z_{i,k}
+\alpha_k\frac{\nabla f_i(x_{i,k})}{n[\mathbf{y}_{i,k}]_i}$.
Fig~\ref{fig1} reports the optimality gap in terms of $F(\tilde x_k)-F(x^\star)$. The results show that Simplified R-Xi-row and Algorithm in \cite{TAC_GradientTrackingBasedDistributedOptimizationWithGuaranteedOptimalitUnderNoisyInformationSharing} are substantially more robust to information-sharing noise than the other baselines. For Xi-row and FRSD augmented with naive diminishing mixing, the decaying gain mitigates the noise effect to some extent, but a non-vanishing bias remains. This is because the information-sharing noise is still accumulated in the gradient-tracking dynamics, although its effect is asymptotically suppressed.
 
\begin{figure}[t]
   \centering
   \includegraphics[scale=0.52]{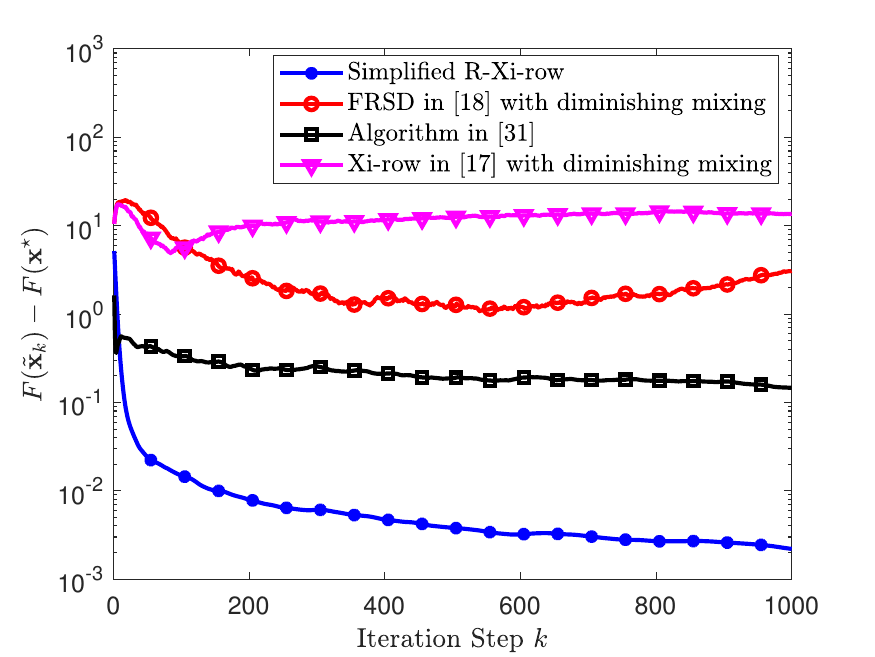}
   \caption{Comparison of the optimality gap $F(\tilde x_k)-F(x^\star)$ under the information-sharing noise $\mathbf{e}_k^x$ and $\mathbf{e}_k^z$.}\label{fig1}
 \end{figure}
\begin{figure}[t]
   \centering
   \includegraphics[scale=0.52]{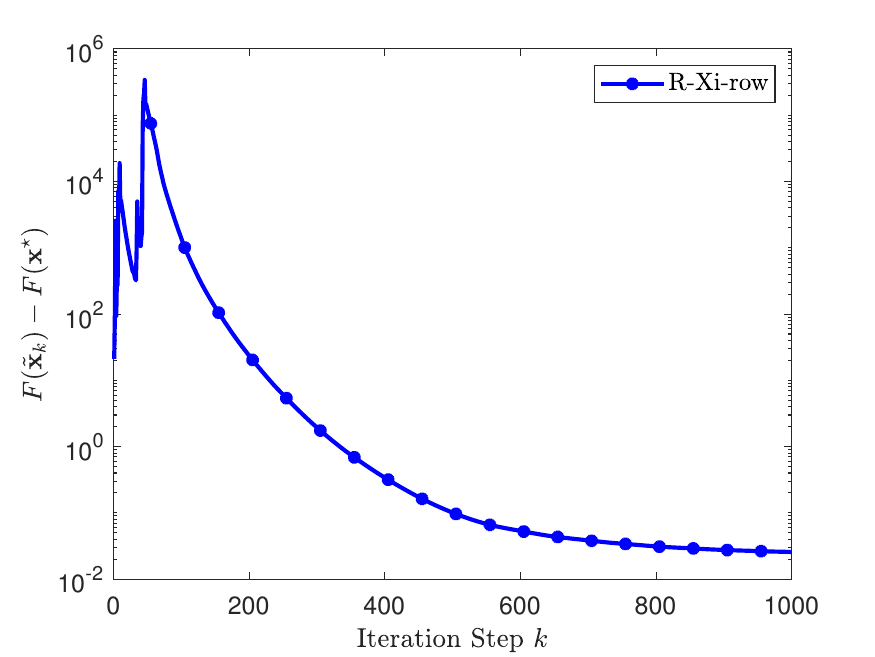}
   \caption{The optimality gap $F(\tilde x_k)-F(x^\star)$ under information-sharing noise $\mathbf{e}_k^x$, $\mathbf{e}_k^y$ and $\mathbf{e}_k^z$.}\label{fig2}
 \end{figure} 
 
In the second scenario, we evaluate R-Xi-row when information-sharing noise affects all the algorithm dynamics, including the left eigenvector estimate. As shown in Fig~\ref{fig2}, the proposed R-Xi-row still effectively attenuates the influence of information-sharing noise in the left eigenvector estimate, although noticeable fluctuations appear in the initial stage. The curves shown in Fig~\ref{fig1} are omitted from Fig~\ref{fig2} because the competing methods diverge at much larger orders of magnitude, which would severely distort the scale of the figure.

\begin{figure}[t]
   \centering
   \includegraphics[scale=0.52]{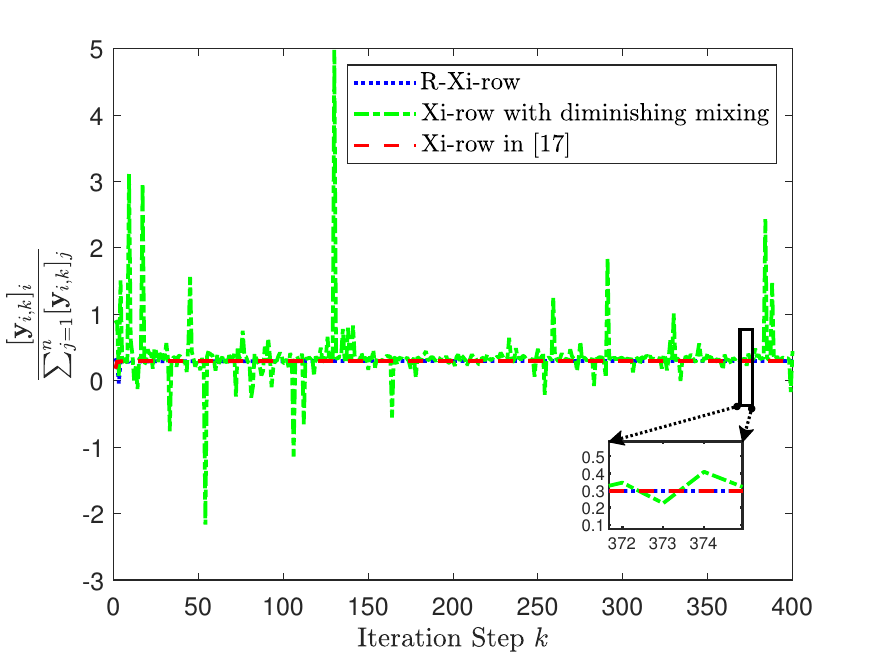}
   \caption{Evolutions of the left eigenvector estimate for agent $1$ under information-sharing noise $\mathbf{e}_k^y$. Xi-row without injected $\mathbf{e}_k^y$ is included as a noiseless reference.}\label{fig3}
 \end{figure}

To further illustrate the effect of information-sharing noise on the left eigenvector estimate, Fig~\ref{fig3} compares the evolution of the quantity $
\frac{[\mathbf{y}_{i,k+1}]_i}{\sum_{j=1}^n [\mathbf{y}_{i,k+1}]_j}$ 
for agent $i=1$. For reference, Xi-row without injected $\mathbf{e}_k^y$ is also included as a noiseless benchmark. The green curve shows that the left eigenvector estimate of Xi-row with diminishing mixing remains persistently affected by the information-sharing noise in the left eigenvector estimate. By contrast, the blue curve demonstrates that the proposed R-Xi-row tracks the noiseless eigenvector estimate closely, and hence asymptotically tracks the left eigenvector of the weight matrix $W$. Moreover, the plotted quantity remains bounded and does not exhibit growth as the suppressed noise term $\lambda_k \mathbf{e}_k^y$ and the correction term $\gamma_kY_0$ accumulate. This observation also explains the scaling gain used in the first branch of \eqref{eq:safeguarded_kappa}, which avoids the vanishing effective step size that would result from directly using $\frac{1}{n[\mathbf{y}_{i,k+1}]_i}$.

\section{Conclusion}\label{sec:conclusion}
In this paper, we investigated distributed optimization over directed graphs with imperfect information sharing. Rather than directly tracking the global gradient, we redesigned the gradient-tracking architecture so that it can be implemented using only row-stochastic weights while reducing the direct accumulation of information-sharing noise in the tracking recursion. Based on this architecture, we developed the R-Xi-row algorithm, which exhibits enhanced robustness to information-sharing noise, including the case where the eigenvector-estimation dynamics are noisy. Numerical experiments demonstrated the advantage of the proposed architecture under noisy optimization and gradient-tracking dynamics, where the conventional row-stochastic methods with diminishing mixing retain noticeable optimization biases. Moreover, when information-sharing noise also affects the left-eigenvector estimation dynamics, R-Xi-row still effectively reduces the optimality gap and recovers the required eigenvector-based gradient scaling. Under strongly convex and $L$-smooth objective functions, we established almost sure convergence of the proposed algorithm under suitable diminishing gains. Moreover, under polynomially decaying gains, we further derived an expected convergence rate result via a Lyapunov recursion in expectation.

\appendices


\section{Proof of \texorpdfstring{Lemma~\ref{y_nr}}{the scaling-gain lemma}}\label[appendix]{app:proof-y-nr}
\begin{proof}
Let's denote the diminishing mixing matrix in \eqref{re_3a} as 
$\overline{W}_k:=(1-\lambda_k)I+\lambda_k W$. 
$Y_{k+1}$ is written by
\begin{eqnarray}\label{recursive_Y}
  \hspace{-14pt}Y_{k+1}\hspace{-7pt}&=&\hspace{-7pt}\overline{W}(k:0)Y_0+\sum_{l=0}^{k}\overline{W}(k:l+1)(\lambda_l\mathbf{e}_{l}^y+\gamma_lY_0),
\end{eqnarray}
where $\overline{W}(k:l):=\overline{W}_k\times\overline{W}_{k-1}\times\cdots\times\overline{W}_{l}$ with $\overline{W}(k:k+1):=I$. Denote $\hat{Y}_k=(I-\mathbf{1}\mathbf{r}^{\top})Y_k$. 
Then, left multiplying both sides of \eqref{recursive_Y} by $I-\mathbf{1}\mathbf{r}^{\top}$ results in
\begin{small}
\begin{eqnarray*}
 \hspace{-3pt} \hat{Y}_{k+1}\hspace{-8pt}&=&\hspace{-8pt}(\overline{W}_k-\mathbf{1}\mathbf{r}^{\top})\hat{Y}_k+\lambda_k(I-\mathbf{1}\mathbf{r}^{\top})\mathbf{e}_k^y+\gamma_k(I-\mathbf{1}\mathbf{r}^{\top})Y_0\nonumber\\
         \hspace{-8pt}&=&\hspace{-8pt}(\overline{W}(k:0)\hspace{-1pt}-\hspace{-1pt}\mathbf{1}\mathbf{r}^{\top})\hat{Y}_0 \hspace{-1pt}+\hspace{-1pt} \sum_{l=0}^{k} (\overline{W}(k\hspace{-1pt}:\hspace{-1pt}l\hspace{-1pt}+\hspace{-1pt}1)\hspace{-1pt}-\hspace{-1pt}\mathbf{1}\mathbf{r}^{\top})\,(\lambda_l \mathbf{e}_l^{y} \hspace{-1pt}+\hspace{-1pt} \gamma_l Y_0).
\end{eqnarray*}
\end{small}

For a matrix
$U=[\mathbf u_1^\top,\ldots,\mathbf u_n^\top]^\top
\in\mathbb R^{n\times d}$, define
$
\|U\|_{\mathbf r}^2
:=\sum_{i=1}^n r_i\|\mathbf u_i\|_2^2$.
By Lemma~1 in \cite{H.Reisizadeh_2023TAC_Distributedoptimizationovertimevaryinggraphswithimperfectsharingofinformation} and $\sqrt{1-x}\leq 1-x/2$, there exist
fixed constants
$\kappa>0$ and $\lambda>0$ such that
\begin{eqnarray}\label{new_1}
  \frac{\|\hat{Y}_{k+1}\|_{\mathbf{r}}}{\Gamma_k} \hspace{-7pt}&\leq&\hspace{-7pt} \frac{\kappa}{\Gamma_k} {\prod_{t=0}^{k} (1-\lambda \lambda_t)} \|\hat{Y}_0\|_{\mathbf{r}}\nonumber\\
  \hspace{-10pt}\hspace{-7pt}&&\hspace{-7pt}
    + \frac{\kappa}{\Gamma_k} \sum_{l=0}^{k} {\prod_{t=l+1}^{k} (1-\lambda \lambda_t)}
    \bigl(\lambda_l \|\mathbf{e}_l^{y}\|_{\mathbf{r}} + \gamma_l \|Y_0\|_{\mathbf{r}}\bigr)\nonumber\\
 \hspace{-7pt}&\overset{(i)}{\leq}&\hspace{-7pt} \frac{\kappa}{\Gamma_k} {\prod_{t=0}^{k} (1-\lambda \lambda_t)} \|\hat{Y}_0\|_{\mathbf{r}}+\frac{\kappa\sqrt{n}\tilde{\sigma}_y}{\Gamma_k\lambda}+\frac{\kappa\sup_{l\leq k}\frac{\gamma_l}{\lambda_l}}{\Gamma_k\lambda},\nonumber\\
 &&
\end{eqnarray}
where $(i)$ follows from Assumption~\ref{ass:noise} and Lemma 4 in \cite{H.Reisizadeh_2023TAC_Distributedoptimizationovertimevaryinggraphswithimperfectsharingofinformation}. Noting that $\lim_{k\rightarrow \infty}\frac{\sup_{0\leq l\leq k}\frac{\gamma_l}{\lambda_l}}{\Gamma_k}= 0$ and $\lim_{k\rightarrow \infty}\Gamma_k=\infty$, taking the limit on both sides of \eqref{new_1} yields, almost surely,
\begin{equation}\label{24_1}
  \lim_{k\rightarrow\infty} \frac{\|\hat{Y}_{k+1}\|_{\mathbf{r}}}{\Gamma_k}=0.
\end{equation}

Meanwhile, from \eqref{recursive_Y}, we have $\mathbf{r}^{\top}Y_{k+1}=\mathbf{r}^{\top}Y_0+\sum_{l=0}^{k}\lambda_l\mathbf{r}^{\top}\mathbf{e}_l^y+\sum_{l=0}^{k}\gamma_l\mathbf{r}^{\top}Y_0$. Recalling that $Y_0=I$, it yields
\begin{eqnarray}\label{25_1}
  \hspace{-10pt} \lim_{k\rightarrow \infty}\frac{\mathbf{r}^{\top}Y_{k}}{\Gamma_k}\hspace{-7pt}&=&\hspace{-7pt}\lim_{k\rightarrow \infty}\frac{\mathbf{r}^{\top}}{\Gamma_k} + \frac{\sum_{l=0}^{k}\lambda_l\mathbf{r}^{\top}\mathbf{e}_l^y}{\Gamma_k} + \frac{\Gamma_k\mathbf{r}^{\top}}{\Gamma_k}=
   \mathbf{r}^{\top}
\end{eqnarray}
almost surely because $\lim_{k\rightarrow \infty}\Gamma_k=\infty$, $\lim_{k\rightarrow \infty}\frac{\Lambda_k}{\Gamma_k}=0$ and $\|\mathbf{e}_{k}^y\|$ is bounded due to Assumption~\ref{ass:noise}. Accordingly, by \eqref{24_1}--\eqref{25_1} and the relation $\hat{Y}_k=(I-\mathbf{1}\mathbf{r}^{\top})Y_k$, it has
\begin{eqnarray}\label{26_1}
  \lim_{k\rightarrow \infty}\frac{Y_k}{\Gamma_k}\hspace{-7pt}&=&\hspace{-7pt}\lim_{k\rightarrow \infty}\frac{\mathbf{1}\mathbf{r}^{\top}Y_k}{\Gamma_k}+\frac{\hat{Y}_k}{\Gamma_k}=\mathbf{1}\mathbf{r}^{\top} \quad \text{a.s.},
\end{eqnarray}
or equivalently, we have 
\begin{equation}\label{26_2}
  \lim_{k\rightarrow \infty}\frac{\mathbf{y}_{i,k}}{\Gamma_k}=\mathbf{r}, \quad
   \forall i\in\mathcal{V}, \quad \text{a.s.}
\end{equation}
Since $\Gamma_k\rightarrow\infty$ and
$\mathbf y_{i,k}/\Gamma_k\rightarrow\mathbf r$ a.s.
with $r_i>0$, we have
$[\mathbf y_{i,k+1}]_i\rightarrow+\infty$ a.s.,
$\forall i\in\mathcal V$.
Therefore, almost surely, there exists a finite 
$T_\kappa\geq0$ such that
$[\mathbf y_{i,k+1}]_i>\varepsilon_\kappa$,
$\forall i\in\mathcal V$ and $k\geq T_\kappa$.
Hence, the safeguard in \eqref{eq:safeguarded_kappa} is inactive
for all sufficiently large $k$, and
$\kappa_{i,k}
=
\frac{\sum_{j=1}^{n}[\mathbf y_{i,k+1}]_j}
{n[\mathbf y_{i,k+1}]_i}$, $\forall k\geq T_{\kappa}$.
Using \eqref{26_2} and $\mathbf 1^\top\mathbf r=1$, it follows that
$\lim_{k\rightarrow\infty}\kappa_{i,k}
=
\frac{\mathbf1^\top\mathbf r}{nr_i}
=
\frac{1}{nr_i}$,
$\forall i\in\mathcal V$ a.s.
This completes the proof.
\end{proof}

\section{Proof of \texorpdfstring{Lemma~\ref{lemma4}}{the conditional-recursion lemma}}\label[appendix]{app:proof-lemma6}
\begin{proof}
The proof is divided into three steps. Let us define $\nabla \tilde{f}_k=\sum_{i=1}^{n}r_i\kappa_{i,k}\nabla f_i(x_{i,k}), \quad \nabla \bar{f}_k=\frac{1}{n}\sum_{i=1}^{n}\nabla f_i(x_{i,k})$.

\textbf{Step 1:} Recursion of $\mathbb{E}[F(\tilde{x}_{k+1})-F(x^*)|\mathcal{F}_k]$.

In view of \eqref{re_3c} and \eqref{re_3b}, $\tilde{x}_{k+1}:=\mathbf{r}^{\top}\mathbf{x}_{k+1}$ reads as 
\begin{eqnarray}\label{17} 
\tilde{x}_{k+1}\hspace{-7pt}&=&\hspace{-7pt} \tilde{x}_k-\alpha_k\nabla \tilde{f}_k +\beta_k \mathbf{r}^{\top} \mathbf{e}_k^{x} -\beta_k\mathbf{r}^{\top}\mathbf{e}_k^{z}. 
\end{eqnarray} 
By \eqref{17} and the $L$-smoothness of $F(\cdot)$, we have 
\begin{eqnarray}\label{F} 
F(\tilde{x}_{k+1}) \hspace{-7pt}&\leq&\hspace{-7pt} F(\tilde{x}_k)-\left\langle \nabla F(\tilde{x}_k), \alpha_k\nabla \tilde{f}_k-\beta_k\mathbf{r}^{\top}\mathbf{e}_k^x+\beta_k\mathbf{r}^{\top}\mathbf{e}_k^z \right\rangle\nonumber\\ 
\hspace{-7pt}& &\hspace{-7pt} +\frac{L}{2}\left\|-\alpha_k\nabla \tilde{f}_k +\beta_k \mathbf{r}^{\top} (\mathbf{e}_k^{x} -\mathbf{e}_k^{z})\right\|_2^2. \end{eqnarray} 
Taking conditional expectation over \eqref{F}, and using Assumption~ \ref{ass:noise} yields 
\begin{eqnarray}\label{E_F} 
\mathbb{E}[F(\tilde{x}_{k+1})|\mathcal{F}_k]\hspace{-7pt}&\leq&\hspace{-7pt} F(\tilde{x}_k)- \alpha_k\langle \nabla F(\tilde{x}_k), \nabla \tilde{f}_k \rangle \nonumber\\ 
\hspace{-7pt}&&\hspace{-7pt}\qquad+\frac{L}{2}\alpha_k^2\|\nabla \tilde{f}_k\|_2^2 +L\sigma\|\mathbf{r}\|_2^2\beta_k^2. 
\end{eqnarray}

We then use the basic inequality $2\langle \mathbf{x},\mathbf{y} \rangle=\|\mathbf{x}\|^2+\|\mathbf{y}\|^2-\|\mathbf{x}-\mathbf{y}\|^2$ $\forall \mathbf{x}, \mathbf{y}\in\mathbb{R}^d$, and rearrange $- \alpha_k\langle \nabla F(\tilde{x}_k), \nabla \tilde{f}_k \rangle$ as 
\begin{align}\label{F_g}
  &-\alpha_k\langle \nabla F(\tilde{x}_k),\nabla \tilde{f}_k \rangle \nonumber\\ 
   \leq&-\frac{\alpha_k}{2}\|\nabla
  F(\tilde{x}_k)\|_2^2 -\frac{\alpha_k}{2}\|\nabla \tilde{f}_k\|_2^2 
   + \alpha_k\|\nabla \bar{f}_k-\nabla F(\tilde{x}_k)\|_2^2 \nonumber\\
   &+ \alpha_k\|\nabla \tilde{f}_k-\nabla \bar{f}_k\|_2^2,
\end{align}
By Lipschitz continuous gradient and Cauchy-Schwarz inequality, $\alpha_k\|\nabla \bar{f}_k-\nabla F(\tilde{x}_k)\|_2^2$ is further expanded by
\begin{align}\label{g-F}
  \hspace{-10pt}\alpha_k\|\nabla \bar{f}_k-\nabla F(\tilde{x}_k)\|_2^2 & \leq \frac{\alpha_kL^2}{n}\sum_{i=1}^{n}\|x_{i,k}-\tilde{x}_k\|_2^2\nonumber\\
   &\leq \frac{\alpha_kL^2}{n}\|\mathbf{x}_k-\tilde{\mathbf{x}}_k\|_{{A}}^2,
\end{align}
where the second inequality follows from Lemma~\ref{lemma1}.

Recalling that $\delta_k:=\max_{i\in\mathcal{V}}|nr_i\kappa_{i,k}-1|$. Applying the Cauchy–Schwarz inequality again to the last term on the right-hand side of \eqref{F_g} yields
\begin{equation}\label{delta_tildef-barf}
  \hspace{-5pt}\|\nabla \tilde{f}_k\hspace{-1pt}-\hspace{-1pt}\nabla \bar{f}_k\|_2^2 = \frac{1}{n^2}\|(\mathbf{1}^{\top}\hspace{-1pt}-\hspace{-1pt}n\mathbf{r}^{\top}\kappa_k)\nabla \mathbf{f}_k\|_2^2\leq \frac{\delta_k^2}{n}\|\nabla \mathbf{f}_k\|_2^2.
\end{equation}

From Th. 2.1.5 in \cite{introduct_nestrov}, it has
$
  \sum_{i=1}^{n}\|\nabla f_i(x^*)-\nabla f_i(\tilde{x}_k)\|_2^2\leq 2nL(F(\tilde{x}_k)-F(x^*))
$,
which further yields
\begin{align}\label{40}
  \hspace{-2pt}\sum_{i=1}^{n}\hspace{-1pt} \|\nabla f_i(\tilde{x}_k)\|_2^2 
   &\hspace{-1pt}\leq\hspace{-1pt} 2\hspace{-1pt}\sum_{i=1}^{n} (\|\nabla f_i(x^*)\hspace{-1pt}-\hspace{-1pt}\nabla f_i(\tilde{x}_k)\|_2^2)\hspace{-1pt}+\hspace{-1pt}\|\nabla f_i(x^*)\|_2^2\nonumber\\
   &\leq 4nL(F(\tilde{x}_k)-F(x^*))+2\sum_{i=1}^{n}\|\nabla f_i(x^*)\|_2^2. 
\end{align}
Therewith, $\|\nabla \mathbf{f}_k\|_2^2$ in \eqref{delta_tildef-barf} is bounded by
\begin{align}\label{f_2}
  \|\nabla \mathbf{f}_k\|_2^2
  &\leq 2 \sum_{i=1}^{n} (\|\nabla f_i(x_{i,k})-\nabla f_i(\tilde{x}_k)\|_2^2+\|\nabla f_i(\tilde{x}_k)\|_2^2)\nonumber\\
  &\leq 2 L^2\sum_{i=1}^{n}\|x_{i,k}-\tilde{x}_k\|_2^2 + 8nL(F(\tilde{x}_k)-F(x^*))\nonumber\\
  &\quad +4\sum_{i=1}^{n} \|\nabla f_i(x^*)\|_2^2
\end{align}

Then subtracting $F(x^*)$ on both sides of \eqref{E_F}, and plugging \eqref{F_g}--\eqref{f_2} into \eqref{E_F}, we arrive at
\begin{small}
\begin{align}\label{F_F^*}
  &\mathbb{E}[F(\tilde{x}_{k+1})-F(x^*)|\mathcal{F}_k] \nonumber\\
  &\leq F(\tilde{x}_k)-F(x^*) -\frac{\alpha_k-L\alpha_k^2}{2}\|\nabla \tilde{f}_k\|_2^2- \frac{\alpha_k}{2}\|\nabla F(\tilde{x}_k)\|_2^2 \nonumber\\
  & \quad + \alpha_k\|\nabla
   \bar{f}_k-\nabla F(\tilde{x}_k)\|_2^2+\alpha_k\|\nabla
   \tilde{f}_k-\nabla \bar{f}_k\|_2^2 + L\sigma\|\mathbf{r}\|_2^2\beta_k^2\nonumber\\
   & \leq\left( 1+8\alpha_kL\delta_k^2 \right) (F(\tilde{x}_k)-F(x^*) ) -\frac{\alpha_k-L\alpha_k^2}{2}\|\nabla \tilde{f}_k\|_2^2\nonumber\\ &\quad-\frac{\alpha_k}{2}\|\nabla F(\tilde{x}_k)\|_2^2  + \frac{\alpha_k L^2+2\alpha_kL^2\delta_k^2}{n} \|\mathbf{x}_k-\tilde{\mathbf{x}}_k\|_{A}^2 \nonumber\\
  &\quad+ \frac{\alpha_k\delta_k^2}{n}\|\nabla \mathbf{f}^*\|_2^2 +L\sigma\|\mathbf{r}\|_2^2\beta_k^2.
\end{align}
\end{small}

\textbf{Step 2:} Recursion of $\mathbb{E}[\|\mathbf{x}_{k+1}-\tilde{\mathbf{x}}_{k+1}\|_{{A}}^2|\mathcal{F}_k]$.

Left multiplying $\Pi_r:=I-\mathbf{1}\mathbf{r}^{\top}$ on the compact form of \eqref{re_3c}, we have the following relation:
\begin{eqnarray}\label{s2_1}
  \hspace{-7pt}&&\hspace{-7pt}\|\mathbf{x}_{k+1}-\tilde{\mathbf{x}}_{k+1}\|_{{A}}^2\nonumber\\ \hspace{-7pt}&=&\hspace{-7pt}\|(W_k-\mathbf{1}\mathbf{r}^{\top})(\mathbf{x}_k-\tilde{\mathbf{x}}_k)+\beta_k\Pi_r\mathbf{e}_k^x-\Pi_{r}\Delta \mathbf{z}_{k+1}\|_{{A}}^2\nonumber\\
  \hspace{-7pt}&=&\hspace{-7pt}\|(W_k-\mathbf{1}\mathbf{r}^{\top})(\mathbf{x}_k-\tilde{\mathbf{x}}_k)-\beta_k\Pi_r A (\mathbf{z}_k-\tilde{\mathbf{z}}_k)\nonumber\\
  \hspace{-7pt}&&\hspace{-7pt} -\alpha_k\Pi_r \kappa_k\nabla \mathbf{f}_k + \beta_k\Pi_r\left( \mathbf{e}_k^x-\mathbf{e}_k^z\right)\|_{{A}}^2.
\end{eqnarray}

Since $\|W_k-\mathbf{1}\mathbf{r}^{\top}\|_{{A}}\leq1-\beta_k\rho_{A}$, it has
\begin{align}\label{s2_2}
  \hspace{-2pt}&\left\| (W_k\hspace{-2pt}-\hspace{-2pt}\mathbf{1}\mathbf{r}^{\top})(\mathbf{x}_k\hspace{-2pt}-\hspace{-2pt}\tilde{\mathbf{x}}_k)\hspace{-1pt}-\hspace{-1pt}\beta_k\Pi_r A (\mathbf{z}_k\hspace{-1pt}-\hspace{-1pt}\tilde{\mathbf{z}}_k)
   \hspace{-1pt}-\hspace{-1pt}\alpha_k\Pi_r \kappa_k\nabla \mathbf{f}_k \right\|_{{A}}^2 \nonumber   \\
   \hspace{-2pt}&\leq (1-\beta_k\rho_A)\|\mathbf{x}_k-\tilde{\mathbf{x}}_k\|_{{A}}^2\nonumber\\
   \hspace{-2pt}&\quad + \frac{1}{\beta_k\rho_A}\left\| \beta_k\Pi_r A (\mathbf{z}_k-\tilde{\mathbf{z}}_k)
   +\alpha_k\Pi_r \kappa_k\nabla \mathbf{f}_k \right\|_{{A}}^2\nonumber\\
   \hspace{-2pt}&\leq (1-\beta_k\rho_A)\|\mathbf{x}_k-\tilde{\mathbf{x}}_k\|_{{A}}^2 +\frac{2\beta_k}{\rho_A}\|\Pi_r A\|_{{A}}^2 \|\mathbf{z}_k-\tilde{\mathbf{z}}_k\|_{{A}}^2\nonumber\\
   \hspace{-2pt}&\quad+ \frac{2\alpha_k^2}{\beta_k\rho_A}\|\Pi_r\kappa_k\|_{{A}}^2 \|\nabla\mathbf{f}_k\|_{{A}}^2,
\end{align}
wherein, the relation  
$\|\mathbf{x}_1+\mathbf{x}_2\|^2\leq (1+\epsilon)\|\mathbf{x}_1\|^2+(1+\frac{1}{\epsilon})\|\mathbf{x}_2\|^2$ is used for any $\mathbf{x}_1$, $\mathbf{x}_2\in\mathbb{R}^d$ and $\epsilon>0$. Specifically, we select $\epsilon=\frac{\beta_k\rho_A}{1-\beta_k\rho_A}$ and $\epsilon=1$ to obtain the first and second inequalities, respectively. 

By taking the expectation on both sides of \eqref{s2_1}, the preceding relationships from \eqref{s2_1}-\eqref{s2_2} together with \eqref{f_2} lead to
\begin{align}\label{s2_5}
  &\mathbb{E}[\|\mathbf{x}_{k+1}-\tilde{\mathbf{x}}_{k+1}\|_{{A}}^2|\mathcal{F}_k]\nonumber\\
  &\leq \left(1-\beta_k\rho_A+\frac{4L^2c_{{A},2}^2\alpha_k^2}{\beta_k\rho_A}\|\Pi_r\kappa_k\|_{{A}}^2\right) \|\mathbf{x}_k-\tilde{\mathbf{x}}_k\|_{{A}}^2 \nonumber\\ 
  &\quad + \frac{16nLc_{A,2}^2\alpha_k^2}{\beta_k\rho_A}\|\Pi_r\kappa_k\|_{{A}}^2 (F(\tilde{x}_k)-F(x^*)) \nonumber\\ 
  &\quad+\frac{2\beta_k}{\rho_A}\|\Pi_r A\|_{{A}}^2\|\mathbf{z}_k-\tilde{\mathbf{z}}_k\|_{{A}}^2\nonumber\\
  &\quad+\frac{8c_{A,2}^2\alpha_k^2}{\beta_k\rho_A}\|\Pi_r \kappa_k\|_{{A}}^2 \|\nabla \mathbf{f}^*\|_2^2  +2\beta_k^2c_{A,2}^2\|\Pi_r\|_{{A}}^2 n\sigma,
\end{align}
wherein, Lemma~\ref{lemma1} is used to transform the relations between the norm $\|\cdot\|_A$ and $\|\cdot\|_2$.

\textbf{Step 3:} Recursion of $\mathbb{E}[\|\mathbf{z}_{k+1}-\tilde{\mathbf{z}}_{k+1}\|_{{A}}^2|\mathcal{F}_k]$.

By \eqref{re_3b}, $\mathbf{z}_{k+1}-\tilde{\mathbf{z}}_{k+1}$ is written by
\begin{eqnarray}\label{theorem2_eq15}
 \hspace{-11pt}&&\hspace{-7pt}\mathbf{z}_{k+1}-\tilde{\mathbf{z}}_{k+1} \nonumber\\ \hspace{-11pt}&=&\hspace{-7pt}{W}_k\mathbf{z}_k+\beta_k\mathbf{e}_k^z+\alpha_k\kappa_k\nabla \mathbf{f}_k-\tilde{\mathbf{z}}_k -\beta_k\mathbf{1}\mathbf{r}^{\top}\mathbf{e}_k^z-\alpha_k\mathbf{1}\mathbf{r}^{\top}\kappa_k\nabla \mathbf{f}_k\nonumber\\
  \hspace{-11pt}&=&\hspace{-7pt}(W_k-\mathbf{1}\mathbf{r}^{\top})(\mathbf{z}_{k}-\tilde{\mathbf{z}}_{k})+\beta_k\Pi_r\mathbf{e}_k^z+\alpha_k\Pi_r \kappa_k\nabla \mathbf{f}_k.
\end{eqnarray}
Further, applying again the contraction bound of $W_k-\mathbf 1 \mathbf r^{\top}$, the above Young-type inequality $\|\mathbf{x}_1+\mathbf{x}_2\|^2\leq (1+\epsilon)\|\mathbf{x}_1\|^2+(1+\frac{1}{\epsilon})\|\mathbf{x}_2\|^2$ with $\epsilon=\frac{\beta_k\rho_A}{1-\beta_k\rho_A}$, and \eqref{f_2}, we obtain
\begin{eqnarray}\label{theorem2_eq17}
  \hspace{-10pt}\mathbb{E}\hspace{-7pt}&&\hspace{-11pt}[\|\mathbf{z}_{k+1}-\tilde{\mathbf{z}}_{k+1}\|_{{A}}^2 |\mathcal{F}_k] 
\leq  \left(1-\beta_k\rho_A\right)\|\mathbf{z}_{k}-\tilde{\mathbf{z}}_{k}\|_{{A}}^2\nonumber\\
  \hspace{-7pt}&&\hspace{-7pt}\quad+\frac{2L^2 c_{{A},2}^2\alpha_k^2 }{\beta_k\rho_A}\|\Pi_r\kappa_k\|_{{A}}^2\|\mathbf{x}_k-\tilde{\mathbf{x}}_k\|_{{A}}^2 \nonumber\\
  \hspace{-7pt}&&\hspace{-7pt}\quad +\frac{8nLc_{A,2}^2\alpha_k^2}{\beta_k\rho_A}\|\Pi_r\kappa_k\|_{{A}}^2 (F(\tilde{x}_k)-F(x^*)) \nonumber\\
  \hspace{-7pt}&&\hspace{-7pt}\quad + \frac{4c_{A,2}^2\alpha_k^2}{\beta_k\rho_A}\|\Pi_r\kappa_k\|_{{A}}^2 \|\nabla \mathbf{f}^*\|_2^2 +\beta_k^2\|\Pi_r\|_{{A}}^2 c_{{A},2}^2 n\sigma.
\end{eqnarray}

At this stage, by combining the inequalities in \eqref{F_F^*}, \eqref{s2_5} and \eqref{theorem2_eq17} together, we have the linear system of inequalities as shown in \eqref{lemma_64}.
\end{proof}

\section{Proof of \texorpdfstring{Lemma~\ref{lem:scalar_recursion_estimate}}{the scalar-recursion lemma}}\label[appendix]{app:proof-scalar}
\begin{proof}
Let $q:=s-a>0$ and $M>0$ such that $M\vartheta>C$. Since the recursion of $\psi_{k}$ is compressed, one has that $\psi_{T_0}\le\frac{M}{(T_0+1)^q}$ holds at some finite time $T_0$. We next prove by induction that $\psi_k\leq\frac{M}{(k+1)^q}, \forall k\ge T$ for some sufficiently large $T$. 

By the mean value theorem applied to \(f(x)=(x+1)^{-q}\), there exists some \(\xi_k\in(k,k+1)\) such that
\begin{eqnarray}\label{mean_value}
\hspace{-10pt}\frac{1}{(k+1)^q}-\frac{1}{(k+2)^q}
\hspace{-7pt}&=&\hspace{-7pt}
q(\xi_k+1)^{-q-1}
\le
\frac{q}{(k+1)^{q+1}}.
\end{eqnarray}

Since $s=q+a<q+1$ and $M\vartheta>C$, there exists $T_1\ge T_0$ such that for all $k\ge T_1$,
\begin{eqnarray}\label{r_q+1}
\frac{M\vartheta-C}{(k+1)^s}\hspace{-7pt}&\ge&\hspace{-7pt}
\frac{Mq}{(k+1)^{q+1}}.
\end{eqnarray}
Therefore, by \eqref{mean_value} and \eqref{r_q+1}, it has that for all $k\ge T_1$, 
\begin{eqnarray*}
\psi_{k+1}\hspace{-7pt}&\le&\hspace{-7pt}
\left(1-\frac{\vartheta}{(k+1)^a}\right)\frac{M}{(k+1)^q}
+\frac{C}{(k+1)^s}\nonumber\\
\hspace{-7pt}&\le&\hspace{-7pt}
\frac{M}{(k+1)^q}
-\left(\frac{M}{(k+1)^q}-\frac{M}{(k+2)^q}\right)=
\frac{M}{(k+2)^q}.
\end{eqnarray*}
This ends the induction and thus the proof is completed.
\end{proof}

\end{document}